\documentclass[12pt]{amsart}
\usepackage[utf8x]{inputenc}
\usepackage{verbatim}
\usepackage{amssymb}
\usepackage{enumerate}
\usepackage{ifthen}

\newtheorem{theorem}{Theorem}[section]
\newtheorem{proposition}[theorem]{Proposition}
\newtheorem{lemma}[theorem]{Lemma}

\newtheorem{corollary}[theorem]{Corollary}
\newtheorem{claim}{Claim}[theorem]
\newtheorem{problem}{Problem}[section]

\newtheorem{conj}{Conjecture}

\theoremstyle{definition}

\newtheorem{definition}[theorem]{Definition}

\theoremstyle{remark}
\newtheorem{remark}{Remark}

\newif\ifdeveloping

\ifdeveloping
\usepackage[notref,notcite]{showkeys}
\fi

\newcommand{\ecal}{{\mathcal E}}

\newcommand{\acal}{{\mathcal A}}
\newcommand{\bcal}{{\mathcal B}}
\newcommand{\ccal}{{\mathcal C}}

\newcommand{\fcal}{{\mathcal F}}
\newcommand{\gcal}{{\mathcal G}}
\newcommand{\hcal}{{\mathcal H}}
\newcommand{\ical}{{\mathcal I}}

\newcommand{\lcal}{{\mathcal L}}

\newcommand{\pcal}{{\mathcal P}}

\newcommand{\scal}{{\mathcal S}}

\newcommand{\ucal}{{\mathcal U}}

\newcommand{\xcal}{{\mathcal X}}

\newcommand{\mc}[1]{\mathcal{#1}}

\newcommand{\cont}{{2^{\omega}}}

\newcommand{\oo}{{{\omega}_1}}
\newcommand{\oot}{{\omega}_2}

\newcommand{\oon}{{\omega}_n}

\newcommand{\as}[1]{A{[#1]}}

\newcommand{\setm}{\setminus}
\newcommand{\empt}{\emptyset}
\newcommand{\subs}{\subset}

\newcommand{\dom}{\operatorname{dom}}
\newcommand{\ran}{\operatorname{ran}}
\def\<{\left\langle}
\def\>{\right\rangle}
\def\cf{\operatorname{cf}}

\def\br#1;#2;{\bigl[ {#1} \bigr]^ {#2} }

\newcommand{\oohp}{\beth_3}

\newcommand{\oopl}{{\omega}_{{\omega}+1}}

\newcommand{\prin}[1]{\bigstar(#1)}

\newcommand{\ccf}{\operatorname{\mbox{${\chi}$}_{\rm CF}}}
\newcommand{\wccf}{\operatorname{w\mbox{${\chi}$}_{\rm CF}}}

\newcommand{\crc}[3]{{\chi}_{\rm CF}(#2,#1,#3)}
\newcommand{\crcw}[3]{w{\chi}_{\rm CF}(#2,#1,#3)}

\newcommand{\carr}[4]{[#2,#1,#3]\to#4 }
\newcommand{\Carr}[4]{[#2,#1,#3]\Rightarrow#4 }
\newcommand{\carrw}[4]{[#2,#1,#3]\to_w#4 }

\newcommand{\ooh}{{\omega}_3}

\newcommand{\km}{{\kappa}^{+m}}

\newcommand{\kmm}{{\kappa}^{+(m-1)}}

\newcommand{\ifomega}[2]{\ifthenelse{\equal{#1}{{\omega}}}
{#1_{\haomega{#2}}}
{#1^{\hamas{#2}}}}

\newcommand{\haomega}[1]{\ifthenelse{\equal{#1}{0}}{}{#1}    }

\newcommand{\hamas}[1]{\ifthenelse{  \equal{#1}{0}  }  {}
{  \ifthenelse  {\equal{#1}{1}} {+} {+#1}}    }

\newcommand{\induw}[5]{[\ifomega{{#1}}{#3},#1,#2,#4]\to_w #5}

\newcommand{\Induw}[5]{[#1,#2,#3,#4]\to_w #5}

\newcommand{\Indu}[5]{[#1,#2,#3,#4]\to #5}

\newcommand{\ninduu}[5]{[#2,#1,#3,#4]\not\to #5}

\newcommand{\nindu}[5]{[\ifomega{#1}{#3},#1,#2,#4]\not\to #5}

\newcommand{\tla}{{2^{\lambda}}}
\newcommand{\ttla}{2^{2^{\lambda}}}

\newcommand{\tip}{\operatorname{tp}}

\newcommand{\MMM}{\mathbf M}

\newcommand{\MM}[3]{\MMM(#1,#2,#3)}
\newcommand{\BBB}{\mathbf B}
\newcommand{\BB}[1]{\BBB(#1)}

\author[A. Hajnal]{Andr\'as Hajnal}
\address
      { Alfr{\'e}d R{\'e}nyi Institute of Mathematics, Budapest, Hungary  }
\email{ahajnal@renyi.hu}

\author[I. Juh\'asz]{Istv\'an Juh\'asz}
\address
      { Alfr{\'e}d R{\'e}nyi Institute of Mathematics, Budapest, Hungary  }
\email{juhasz@renyi.hu}

\author[L. Soukup]{Lajos Soukup}
\address
      { Alfr{\'e}d R{\'e}nyi Institute of Mathematics, Budapest, Hungary  }
\email{soukup@renyi.hu}
\urladdr{http://www.renyi.hu/$\tilde{}$soukup}

\author[Z. Szentmikl\'ossy]{
Zolt\'an Szentmikl\'ossy}
\address{E\"otv\"os University of Budapest}
\email{zoli@renyi.hu}

\subjclass[2000]{03E35, 03E05}
\keywords{coloring, conflict-free coloring, almost disjoint,
  essentially disjoint}
\title[Conflict free colorings]
   {Conflict free colorings of
(strongly) almost disjoint set-systems}
\thanks{The preparation of this paper was partially
supported by  OTKA grants K 61600 and K 68262.}

\date{March, 2010.}

\begin{document}

\begin{abstract}
$f:\cup\acal\to {\rho}$ is called a {\em conflict free coloring of
the set-system  $\acal$ (with ${\rho}$ colors)} if
\begin{displaymath}\forall A\in \acal\,\, \exists\, {\zeta}<{\rho}\, (\,|A\cap
f^{-1}\{{\zeta}\}|=1\,).
\end{displaymath}
The {\em  conflict free chromatic number} $\ccf(\acal)$ of $\acal$
is the smallest $\rho$ for which $\acal$ admits a conflict free
coloring with ${\rho}$ colors.

$\acal$ is a $(\lambda,\kappa,\mu)$-system if $|\acal| = \lambda$,
$|A| = \kappa$ for all $A \in \acal$, and $\acal$ is
${\mu}$-almost disjoint, i.e. $|A\cap A'|<{\mu}$ for distinct $A,
A'\in \acal$. Our aim here is to study $$\ccf(\lambda,\kappa,\mu)
= \sup \{\ccf(\acal): \acal \mbox{  is a }
(\lambda,\kappa,\mu)\mbox{-system} \}$$ for $\lambda \ge \kappa
\ge \mu$, actually restricting ourselves to $\lambda \ge \omega$
and $\mu \le \omega$.

For instance, we prove that

\smallskip

\begin{itemize}
\item for any limit cardinal $ \kappa$ (or $\kappa = \omega$) and integers\\ $n \ge 0,\,k >
0$, GCH implies
\[
\ccf(\kappa^{+n},t,k+1) =\left\{
\begin{array}{lll}
\kappa^{+(n+1-i)}&\mbox{ if $\,\,i\cdot k < t \le (i+1)\cdot k\,,$}
\\&\makebox[60pt]{}\mbox{$i = 1,...,n$;}\\
{}\\
\kappa& \mbox{ if $\,\,(n+1)\cdot k < t\,$}\,;
\end{array}
\right.
\]

\smallskip

\item if $\lambda \ge \kappa \ge \omega > d > 1\,,$ then
$\,\lambda < \kappa^{+\omega}$ implies $\ccf(\lambda,\kappa,d) <
\omega$\\ and $\lambda \ge \beth_\omega(\kappa)\,$ implies
$\ccf(\lambda,\kappa,d) = \omega\,$;

\smallskip

\item GCH implies $\,\ccf(\lambda,\kappa,\omega) \le \omega_2$ for $\lambda \ge \kappa \ge
\omega_2$ and\\V=L implies $\,\ccf(\lambda,\kappa,\omega) \le
\omega_1$ for $\lambda \ge \kappa \ge \omega_1\,$;

\smallskip

\item the existence of a supercompact cardinal implies\\ the consistency of GCH plus\\
$\ccf(\aleph_{\omega+1},\omega_1,\omega)= \aleph_{\omega+1}$  and\\
$\ccf(\aleph_{\omega+1},\omega_n,\omega) = \omega_2$ for $2 \le n
\le \omega\,$ ;

\smallskip

\item CH implies $\,\ccf(\omega_1,\omega,\omega) = \ccf(\omega_1,\omega_1,\omega) =
\omega_1$, while\\ $MA_{\omega_1}$ implies
$\,\ccf(\omega_1,\omega,\omega) = \ccf(\omega_1,\omega_1,\omega) =
\omega\,$.
\end{itemize}
\end{abstract}

\maketitle

\section{Introduction}

If $\acal$ is a set-system and $\rho$ is a cardinal then a
function $f : \cup \acal \to \rho$ is called a {\em proper
coloring} of $\acal$ with ${\rho}$ colors if $f$ takes at least
$2$ values on each $A \in \acal$. The smallest $\rho$ for which
$\acal$ admits a proper coloring with $\rho$ colors is the {\em
chromatic number of} $\acal$ and is denoted by $\chi(\acal)$. The
chromatic numbers of various set-systems, in particular almost
disjoint ones, had been systematically studied by Erd{\H o}s and
Hajnal and others in \cite{EH3}, \cite{EH1}, and \cite{EH2}.

A function $f:\cup\acal\to {\rho}$  is called a {\em conflict free
coloring} of $\acal$ with ${\rho}$ colors if
\begin{displaymath}\forall A\in \acal\ \exists {\zeta}<{\rho}\ (|A\cap
f^{-1}\{{\zeta}\}|=1).
\end{displaymath}
We say that $f$ is a {\em  weak conflict free coloring } of
$\acal$ if in the above definition the assumption $\dom
(f)=\cup\acal$ is weakened to $\dom(f)\subs\cup\acal$.

The {\em conflict-free chromatic number} and the {\em weak
conflict-free chromatic number} of a set-system $\acal$, denoted
by $$ \ccf(\acal) \mbox{ and }\;w\ccf(\acal)$$ respectively, are
defined as the minimum number of colors needed for a conflict free
or a weak conflict-free coloring of $\acal$, respectively.

Conflict-free colorings of hypergraphs, that is of systems of
finite sets, were first studied in Cheilaris \cite{CHE} and
Pach-Tardos \cite{pach}. Earlier, conflict-free colorings were
mainly considered for some concrete hypergraphs, usually defined
by geometric means \cite{E}. J\'anos Pach  suggested to us that it
would be worth while to study the conflict free colorings of
almost disjoint transfinite set systems. It took little time to
convince us.

Before going on with the story we state a few very elementary
facts. Note first that $\chi(\acal)$ is only defined if every
member of $\acal$ has at least two elements, so from here on this
is assumed for every set-system $\acal$.

\begin{proposition}
\begin{enumerate}[$(1)$]
\item $\chi(\acal) \leq \ccf(\acal)\leq {w\ccf(\acal)+1}$.
\medskip
\item $\chi(\acal)=\ccf(\acal) $ provided $|A|\leq 3$ for all
$A\in\acal$.
\medskip
\item For each $\kappa \geq \omega$ there exists a quadruple system
$\acal$ with $\chi(\acal)=2$ and $\ccf(\acal)= \kappa$.
\end{enumerate}
\end{proposition}

\begin{proof}
The first statement is trivial, the second follows from $2+2>3$.
To see the third, let
\begin{displaymath}
 \acal= \{H\in  \br \kappa;4;:\text{ $H$ contains
two even and two odd ordinals}\}.
 \end{displaymath}
\end{proof}

For any cardinals $ {\mu}$ and ${\nu}$, the set system  $\acal$ is
called {\em $({\mu},{\nu})$-almost disjoint} if
\begin{displaymath}
|\cap\bcal|<\mu
\end{displaymath}
whenever $\bcal \in  \br \acal;{\nu};$. We simply write {\em
$\mu$-almost disjoint} instead of $({\mu}, 2)$-almost disjoint.

A graph $G=\<V,E\>$ is called {\em $({\mu},{\nu})$-almost
disjoint} iff the family $\{E(v):v\in V\}$ is
$({\mu},{\nu})$-almost disjoint, where $E(v)=\{w\in V: \{v,w\} \in
E\}$. In \cite{EH1}, Erd\H {o}s and Hajnal proved, in 1966, that
if $n <\omega$ and  $G$ is an $(n,\omega_1)$-almost disjoint
graph, then $\chi(G)\leq\omega$, which of course means $\chi(E)
\le \omega$. They tried to state a generalization of this result
for set-systems consisting of finite sets, but failed. Such a
generalization was found in the triple paper \cite{EH2} with
B.Rothchild, where some results were proved for finitary
$({\mu},{\nu})$-almost disjoint set-systems. In Part I we prove
results for such set-systems that are improvements of the results
of \cite{EH2}. The work started in \cite{EH2} was continued in the
almost ninety page long triple paper \cite{EGH} of Erd\H os,
Galvin and Hajnal. Although we could find some improvements of the
results of this paper as well, we did not dare to start to
investigate this methodically.

Our main objects of study will be the (weak) conflict free
chromatic numbers of $(\lambda,\kappa,\mu)$-systems: $\acal$ is a
$(\lambda,\kappa,\mu)$-system if $|\acal| = \lambda$, $|A| =
\kappa$ for all $A \in \acal$, and $\acal$ is ${\mu}$-almost
disjoint. We shall always assume that $\lambda \ge \kappa \ge \mu$
and that $\lambda$ is infinite. These assumptions imply that if
$\acal$ is a $(\lambda,\kappa,\mu)$-system then $|\cup \acal| \le
\lambda$, hence $\acal$ has an isomorphic copy $\bcal \subs
[\lambda]^\kappa$. Conversely, if $\mu < \omega$ then for every
$\mu$-almost disjoint $\acal \subs [\lambda]^\kappa$ we have
$|\acal| \le \lambda$.

Now, our basic definition is the following. Let $\psi$ be any one
of the functions $\chi,\, \ccf\,$, or $w\ccf$.

\begin{definition}\label{def:psi}
For $\lambda \ge \kappa \ge \mu$ we set
$$\psi(\lambda,\kappa,\mu)
= \sup \{\psi(\acal): \acal \mbox{  is a }
(\lambda,\kappa,\mu)\mbox{-system} \}.$$
\end{definition}

Let us point out certain basic properties of these. First, it is
obvious that $\,\chi(\lambda,\kappa,\mu) \le
\ccf(\lambda,\kappa,\mu)\,$ and
$$w\ccf(\lambda,\kappa,\mu) \le \ccf(\lambda,\kappa,\mu) \le
w\ccf(\lambda,\kappa,\mu)+1\,.$$ Thus, although in some cases
$w\ccf(\lambda,\kappa,\mu)$ is much easier to handle than
$\ccf(\lambda,\kappa,\mu)$, the results on the former reveal a lot
of information about the latter. Second, it is immediate from
their definitions that they are monotone increasing in their first
and third variables.

Intuitively, it also seems plausible that they are monotone
decreasing in their second variable: the larger the sets, the more
room we have to color them appropriately. For
$\chi(\lambda,\kappa,\mu)$ this is obvious and all our results
confirm this for the other two as well. Alas, we do not have a
formal proof of this, so we propose it as a conjecture.

\begin{conj}\label{co:conj}
If $\lambda \ge \kappa > \kappa' \ge \mu$ with $\lambda$ infinite,
then $$\ccf(\lambda,\kappa,\mu) \le \ccf(\lambda,\kappa',\mu)\,.$$
\end{conj}

Third, we note that if $\mu = 1$, i.e. we deal with {\em disjoint}
systems, then trivially $w\ccf(\lambda,\kappa,1) = 1$ and
$\chi(\lambda,\kappa,1) = \ccf(\lambda,\kappa,1) = 2$.
Consequently, in what follows we always assume $\mu \ge 2$.

While working on this paper we found it useful to write
$\carr{\kappa}{\lambda}{\mu}{\rho}$ for the relation
$\ccf(\lambda,\kappa,\mu) \le \rho$ and, analogously,
$\carrw{\kappa}{\lambda}{\mu}{\rho}$ for the relation
$w\ccf(\lambda,\kappa,\mu) \le \rho$.

On one hand, the behavior of these symbols shows much similarity
to the symbol $M(\lambda,\kappa,\mu)\rightarrow B(\rho)$,
investigated in \cite{EH3}, \cite{HJS1}, and \cite{HJS2}, meaning
that every $(\lambda,\kappa,\mu)$-system has a $\rho$-transversal,
i.e. a set $B$ that meets every element of $\acal$ in a non-empty
set of size $<{\rho}$. But the main reason for this apparent
duplication of our notation is that certain variations of these
arrow relations will turn out to be quite useful later.

The paper is naturally divided into three parts as follows:

\smallskip

\noindent Part I.  $\lambda \ge \omega > \kappa \ge \mu$,

\smallskip

\noindent Part II.  $\lambda \ge \kappa \ge \omega > \mu$,

\smallskip

\noindent Part III.  $\lambda \ge \kappa \ge \omega = \mu$,

\smallskip
\noindent and the three parts are largely independent of each
other. However closure arguments, in the ``modern" disguise of
elementary chains, have been extensively used in all three parts.
This method was developed in the papers \cite{MIL,EH3, HJS1,HJS2},
the earlier ones naturally using different terminology.

The main result of Part I is theorem \ref{tm:gch} that gives a
full description of $\ccf(\lambda,\kappa,\mu)$ for this case in
which $\kappa$ (and hence $\mu$) is finite. We also have ZFC
results, for instance corollary \ref{cor:w1} that states
$\ccf(\lambda,2k,k+1) = \lambda$ for any $\lambda \ge \omega$ and
$0 < k < \omega$. Of course, then conjecture \ref{co:conj} would
imply $\ccf(\lambda,t,k+1) = \lambda$ for $k < t < 2k$  as well.
In corollary \ref{cor:odd} we could prove this, with some effort,
for ``almost all" $\lambda$, namely those that are not successors
of singular cardinals.

In Part II we first show that $\ccf(\lambda,\kappa,d)$ is always
countable, i.e. $[\lambda,\kappa,d] \to \omega$ holds, if $\kappa
\ge \omega
> d$. In fact we show something stronger that involves a modified
arrow relation. To get this we first need the following notation.

\begin{definition}
If $f$ is a function and $A$ is any set, we let
$$f[A]=\{f({\alpha}):{\alpha}\in A\cap \dom(f)\}$$  and

$$I_f(A)=\{{\xi} \in \ran(f): |A \cap f^{-1}\{\xi\}|=1\}.$$
\end{definition}
Thus, $f$ is a weak conflict free coloring of a set system $\acal$
exactly if $I_f(A) \ne \emptyset$ for all $A \in \acal$. Keeping
this in mind, we indeed define a strengthening of the relation
$[\lambda,\kappa,\mu] \to \rho$ below.

\begin{definition}
Assume that $\lambda \ge \kappa \ge \rho \ge \omega$ and $\mu \le
\kappa$. Then $[\lambda,\kappa,\mu] \Rightarrow \rho$ denotes that
there is a function $f : \cup \acal \to \rho$ such that $|\rho
\setm I_f(A)| < \rho$ holds for all $A \in \acal$.
\end{definition}

What we actually prove in theorem \ref{tm:korlatos} is
$[\lambda,\kappa,d] \Rightarrow \omega$ whenever $\kappa \ge
\omega > d$.

\medskip
In \cite{EH3} it was proved that $M(\kappa,\kappa^{+n},d)
\rightarrow B((n+1)(d-1)+2)$ and that this is best possible
assuming GCH. In Sections  5, 6, and 7 of Part II we prove
analogous results for our symbols. In some sense, these chapters
are the heart of our present paper. The results and their proofs
seem more complicated than those from Part I, and there are a
number of unsolved problems left.

By theorem \ref{tm:egeszresz}, if $m$ and $d$ are natural numbers
and $\kappa$ is infinite, then
\begin{equation}\notag
w\ccf(\kappa^{+m},\kappa,d) \le {\left\lfloor
{\frac{(m+1)(d-1)+1}2} \right\rfloor+1}.
\end{equation}
>From the other side, theorems \ref{tm:step3'} and \ref{tm:step3}
yield $$w\ccf(\beth_m(\kappa),\kappa,2) \ge \left\lfloor \frac
{m}{2}\right\rfloor+2\,$$ and
$$w\ccf(\beth_m(\kappa),\kappa,2\ell+1) \ge (m+1)\cdot \ell +
1\,,$$ respectively.  Consequently, under GCH we get the exact
values
\begin{displaymath}
\crcw {\kappa}{{\kappa}^{+m}}2={\lfloor m/2\rfloor +2}
\end{displaymath}
and
\begin{displaymath}
\crcw {\kappa}{{\kappa}^{+m}}{2\ell+1}=(m+1)\cdot\ell +1.
\end{displaymath}

It seems to be much more challenging to find the exact values of,
say,  $\,\crc {\omega}{\omega_m}d\,$, even under GCH and for $d =
2$. We conjecture that GCH implies $\crc {\kappa}{\lambda}d= \crcw
{\kappa}{\lambda}d+1$, but we could not even prove that $$\crc
{\omega}{{\omega_m}}2= \lfloor m/2\rfloor +3\,\,$$ holds for each
$m\in {\omega}$. This equality  holds  for $m=0,1$ in ZFC, by
proposition \ref{f:omega}, and for $m=3$ under GCH , by theorem
\ref{tm:ooh}. However, for $m=2$, we cannot prove even the
consistency of $\crc {\omega}{\oot}2=4$.

\bigskip

In Part III we only investigate conflict free colorings of
$(\lambda,\kappa,\omega)$-systems, but it is fairly clear that
most of the results would generalize for arbitrary infinite
cardinals $\mu$ instead of $\omega$. This practically means that
we only follow in the footsteps of the triple paper \cite{HJS1},
leaving the cases covered only in \cite{HJS2} alone. Results for
these cases are reserved for later publications or left for future
generations.

\medskip

By a result of Komj\'ath \cite{KO}, we have
$\chi(2^\omega,\omega,\omega) = \crc
{\omega}{2^{\omega}}{\omega}={2^{\omega}}$, and if
$\clubsuit(\lambda)$ holds for a regular $\lambda$ then $\crc
{\omega}\lambda{\omega}=\lambda$. So, in ZFC, we can not have any
non-trivial upper bound for $\crc {\omega}\lambda{\omega}$. By
theorem \ref{tm:ch_w1w1}, CH implies
$\ccf(\omega_1,\omega_1,\omega) = \omega_1$, so even for
uncountable $\kappa$ we expect to have only uncountable upper
bounds for $\ccf(\lambda,\kappa,\omega)$.

Such bounds can indeed be found, at least consistently. For
instance, theorem \ref{tm:above_oot} says that if $\mu^\omega =
\mu$ holds for each $\mu < \lambda$ with $\cf(\mu) = \omega$, then
we have $[\lambda,\kappa,\omega] \Rightarrow \omega_2$, hence
$\ccf(\lambda,\kappa,\omega) \le \omega_2$, whenever $\oot\le
{\kappa}\le {\lambda}$. Moreover, if in addition we also assume
$\Box_\mu$ for all $\mu$ with $\omega = \cf(\mu) < \mu < \lambda$,
then $\ccf(\lambda,\kappa,\omega) \le \omega_1$ whenever $\omega_1
\le {\kappa}\le {\lambda}$, by theorem \ref{tm:above_oo}.

These results are very sharp, at least modulo large cardinals.
Indeed, we show in section 9 that the existence of a supercompact
cardinal implies the consistency of GCH plus the following two
equalities:
    \begin{itemize}
     \item $\crc \oo\oopl{\omega}=\oopl$,

\smallskip

     \item $\crc\oon\oopl{\omega}=\oot\,$ for  $\,2 \le n\le{\omega}$.
    \end{itemize}

\smallskip

We close each Part by stating the problems that are nagging us
most.

\medskip

Our notation is standard, as e.g. in \cite{KU} . If ${\lambda}$ is
an infinite cardinal then we call {\em a ${\lambda}$-chain of
elementary submodels} a continuous sequence
$\<N_{\alpha}:{\alpha}<{\lambda}\>$ such that $N_0=\empt$,
$\{N_{\alpha}:1\le {\alpha}<{\lambda}\}$ are elementary submodels
of  $\<H_{\theta},\in \>$ for some fixed, appropriately chosen
regular cardinal $\theta$, moreover $|N_{\alpha}|<{\lambda}$,
$N_{\alpha}\in N_{{\alpha}+1}$ and ${\alpha}\subs N_{\alpha}\cap
{\lambda}$ for ${\alpha}<{\lambda}$. If ${\lambda}={\kappa}^+$
then we also assume ${\kappa}\subs N_1$. We put $N_0 = \empt$ to
ensure that $\{N_{{\alpha}+1}\setm
N_{\alpha}:{\alpha}<{\lambda}\}$ be a partition of
$\cup\{N_{\alpha}:{\alpha}<{\lambda}\}$.

\bigskip

\bigskip

\begin{center}
{\bf {\sc Part I. The case $\lambda \ge \omega > \kappa \ge \mu$}}
\end{center}
\nopagebreak
\section{Upper bounds}\label{sc:fin}

It is obvious that for every $\acal \subs \mathcal{P}(\kappa)$ we
have $\ccf({\acal})\le{\kappa}$. Our next result shows that this
inequality remains true for suitably almost disjoint families
$\acal$ of finite subsets of $\kappa^{+n}$ with $\omega > n > 0$,
provided that the members of $\acal$ are large enough.

\begin{theorem}\label{tm:ub}
Let $\kappa \ge \nu \ge \omega$ where $\nu$ is assumed to be
regular, moreover $n \ge 1$ and $k\ge 1$ be natural numbers. If
$\acal$ is a $(k+1,{\nu})$-almost disjoint subfamily of
$[\kappa^{+n-1}]^{<\omega}\, $ such that $|A| > n\cdot k$ for
every $ A \in \acal$, then $\ccf({\acal})\le{\kappa}$.
\end{theorem}

\begin{proof}
We actually prove the following stronger statement $(*)_n$ by
induction on $n \ge 1$, keeping all the other parameters fixed.

\begin{itemize}
\item[$(*)_n$]
If
 $\,\acal\subs \br \kappa^{+n-1};< \omega;\,\setminus\,[\kappa^{+n-1}]^{\le n\cdot k}\,$
 is
$(k+1,{\nu})$-almost disjoint and $g:\acal\to \br
{\kappa};<{\nu};$ then there is a function $f:\kappa^{+n-1} \to
\kappa\,$ such that
\begin{displaymath}
I_f(A)\setm g(A)\ne \empt
\end{displaymath}
for each $A\in \acal$.
\end{itemize}

\noindent {\bf First step: $n=1$.}

We define an injective function $f:\kappa \to \kappa$ inductively
on $\xi < \kappa$. Assume that we have defined $f\restriction \xi$
and let
$$\acal_\xi =\{A\in \acal: \xi =\max A\}\,.$$ Clearly,
$|\acal_\xi| <  \kappa$, hence $|\bigcup\{g(A):A\in \acal_\xi \}|
< \kappa$ as well. The second inequality uses that $\kappa$ is
regular in case $\nu = \kappa$ and is trivial otherwise. Thus we
may pick
\begin{displaymath}
f(\xi)\in {\kappa}\setm (f[\,\xi]\cup\bigcup\{g(A):A\in \acal_\xi
\})\,.
\end{displaymath}
By the construction, we have $f(\max A)\in I_f(A)\setm g(A)$ for
all $A\in \acal$. (Of course, this construction does not make use
of the almost disjointness or the largeness assumptions made on
$\acal$.)
\medskip

\noindent {\bf Inductive step: $(*)_n\to (*)_{n+1}$.}

Now we start with a $(k+1,{\nu})$-almost disjoint system
$$\,\acal\subs \br \kappa^{+n};<
\omega;\,\setminus\,[\kappa^{+n}]^{\le (n+1)\cdot k}\,$$ and a
function $g:\acal\to \br {\kappa};<{\nu};$. Let us then fix a
${\kappa}^{+n}$-chain of elementary submodels
$\<N_{\alpha}:{\alpha}<{\kappa}^{+n}\>\,$ with $\acal,g \in N_1$.
For every ${\alpha}<{\kappa}^{+n}$ let $Y_{\alpha}= {\kappa}^{+n}
\cap ( N_{{\alpha}+1}\setm N_{\alpha})\,$,
$\acal_{\alpha}=\acal\cap (N_{{\alpha}+1}\setm N_{\alpha})$ and,
finally, $\acal'_{\alpha}=\{A\cap Y_{\alpha}:A\in
\acal_{\alpha}\}.$ We may clearly assume that $|N_{{\alpha}+1}| =
|Y_\alpha| = {\kappa}^{+n-1}$ for all $\alpha < {\kappa}^{+n}$.

For every $A\in \acal\setm  N_{\alpha}$ we have $|A\cap
N_{\alpha}| \le k$  because $\acal$ is $(k+1,{\nu})$-almost
disjoint. So if $A\in \acal_{\alpha}$ then $|A\cap Y_{\alpha}| >
(n+1)\cdot k - k = n\cdot k$, consequently $\acal'_{\alpha}\subs
[Y_\alpha]^{<\omega} \setminus \br Y_{\alpha};\le n\cdot k ;$ and,
clearly, $\acal'_{\alpha}$ is $(k+1,{\nu})$-almost disjoint.

We next define, for each ${\alpha}<{\kappa}^{+n}$, a function
$f_{\alpha}:Y_{\alpha}\to {\kappa}$, using transfinite induction
as follows. Assume that $f_{\xi}$ has been defined for each
${\xi}<{\alpha}<{\kappa}^{+n}$ and set
$f_{<{\alpha}}=\cup\{f_{\xi}:{\xi}<{\alpha}\}$. For any $A'\in
\acal'_{\alpha}$ let
\begin{displaymath}
g_{\alpha}(A')= \bigcup\{f_{<{\alpha}}[A]\cup
g(A):A\in\acal_{\alpha}\land A\cap Y_{\alpha}=A'\}.
  \end{displaymath}
Since $|A'| > n\cdot k \ge k$ (recall that $n \ge 1\,$!) and
$\acal$ is $(k+1,{\nu})$-almost disjoint,
$|\{A\in\acal_{\alpha}:A\cap Y_{\alpha}=A'\}| < \nu$ and hence
$g_{\alpha}(A') \in [\kappa]^{<\nu} $, using that  $\nu$ is
regular.

Thus, the inductive assumption $(*)_n$ can be applied to
$\acal_{\alpha}'$ and $g_{\alpha}$ and yields us a function
$f_{\alpha}:Y_{\alpha} \to {\kappa}$ such that
\begin{displaymath}
  I_{f_{\alpha}}(A')\setm g_{\alpha}(A')
\ne \empt
\end{displaymath}
for each $A'\in \acal'_{\alpha}$.

Finally, let $f=\cup\{f_{\alpha}:{\alpha}<{\kappa}^{+n}\}$. Then
for every $A\in \acal_{\alpha}$ we have
$$
 I_f(A)\setm g(A)\supset
I_{f_{\alpha}}(A\cap Y_{\alpha})\setm g_{\alpha}(A\cap
Y_{\alpha})\ne \empt,$$ hence we are done because $\acal = \bigcup
\{\acal_\alpha : {\alpha < {\kappa}^{+n}}\}$.
\end{proof}

\bigskip

We now give a consistency result in the spirit of theorem
\ref{tm:ub} that uses Martin's axiom.

\begin{theorem}\label{tm:ubma}
Assume $MA_\lambda(K)$, i.e. $MA_\lambda$ for partial orders
satisfying property $K$. Then for every natural number $k$ and for
every $(k+1,\omega)$-almost disjoint system $\acal \subs
[\lambda]^{<\omega}$ such that $|A| > 2k$ for all $A \in \acal$ we
have $\ccf({\acal})\le{\omega}$.
\end{theorem}

\begin{proof}

We first define the poset $\pcal_\acal=\<P_\acal,\le\>$ as
follows: A function $f\in Fn({\lambda},{\omega})$ (that is a
finite partial function from $\lambda$ to $\omega$) is in
$P_\acal$ iff $I_f(A)\ne \empt$ whenever $A\in \acal$ and $A \subs
\dom(f)$. We then let $f\le g$ iff $f\supset g$.

We claim that the poset $\pcal_\acal$ satisfies property $K$.
Indeed, assume that $\{f_{\alpha}:{\alpha}<\oo\} \subs P_\acal$.
Without loss of generality we can assume that
\begin{equation}\label{eq:delta1}
\text{$f_{\alpha}=f\cup^* f_{\alpha}'$, and  $\dom
(f'_{\alpha})\cap \dom (f'_{\beta})=\empt$ for $\{{\alpha},
{\beta}\} \in [\oo]^2$.}
\end{equation}

For each ${\alpha}<\oo$ then $\acal_{\alpha}=\{A\in \acal:|A\cap
\dom (f_{\alpha})| > k\}\,$ is finite because $\acal$ is
$(k+1,\omega)$-almost disjoint. Let
$$F({\alpha})=\{{\beta}<\omega_1:\dom(f_{\beta}')\cap \cup \acal_{\alpha}\ne
\empt\}\,,$$ then $F({\alpha})$ is also finite. So by the
(simplest case of the) free set theorem for set mappings we can
find a set $S \in [\omega_1]^{\omega_1}$ such that ${\alpha}\notin
F({\beta})$ and ${\beta}\notin F({\alpha})$ whenever
$\{\alpha,\beta \} \in [S]^2$.

We claim that $f=f_{\alpha}\cup f_{\beta}\in P_\acal$, hence
$f_\alpha$ and $f_\beta$ are compatible, for any such pair
$\{\alpha,\beta \}$. By (\ref{eq:delta1}), $f$ is a function. So
assume now that $A\in \acal$ with $A\subs \dom (f)$. Since
$|A|>2k$  we can assume that e.g. $|A\cap \dom (f_{\alpha})|
> k$, that is $A\in \acal_{\alpha}$, hence $A\cap \dom
(f'_{\beta})=\empt$. But then $A\subs \dom (f_{\alpha})$ and so
$I_f(A)=I_{f_{\alpha}}(A)\ne \empt$. Thus $f\in P_\acal$,
completing the proof that $P_\acal$ has property $K$.

The rest of the proof is a standard  density argument that we
leave to the reader.
\end{proof}

{\bf Remark:} A slightly weaker statement than theorem
\ref{tm:ubma}, for the chromatic number $\chi$ instead of the
conflict free chromatic number$\ccf$, was proved in \cite[Theorem
5.6]{EGH}. It was asked there, in Problem 2, if the statement
remains true for $(k,\omega_1)$-almost disjoint families. We still
do not know the answer to this.

\section{Lower bounds}

We start this section with presenting a result which implies that
the assumptions on the set systems formulated in theorems
\ref{tm:ub} and \ref{tm:ubma}, namely that their members should be
``suitably large", are really necessary.


\begin{theorem}\label{tm:lb}
Assume that $\lambda \ge \omega$ and $\mu$ are cardinals, $n \ge
2$, $k \ge 1$ are natural numbers such that the partition relation
$$\lambda \to (n)_{\mu^k}^{n-1}$$ holds true. (Of course, if $\mu$
is infinite then $\mu^k = \mu$.) Then we have $\ccf(\lambda,t,k+1)
> \mu$ for every number $\,t$ satisfying $k < t \le n\cdot k$ if $n >
2$ and for every {\em even} number $t$ satisfying $k < t \le
2\cdot k$ if $\,n = 2$.
\end{theorem}

\begin{proof}

Let us put $H = [\lambda]^{n-1} \times k$, then $|H| = \lambda$.
We shall construct a $(k+1)$-almost disjoint family $\acal \subs
[H]^t$ of cardinality $\lambda$ which does not have  a conflict
free coloring with $\mu$ colors.

For each $Y \in [\lambda]^n$ we may choose a $t$-element set $A_Y
\in \big[[Y]^{n-1}\times k \big]^t$ such that for every $i<k$ we
have
\begin{equation}\label{eq:B}
|\{B\in \br Y;n-1;:\<B,i\>\in A_Y\}|\ne 1.
\end{equation}
This is easy to check and this is the point where $t$ has to be
even in case $n = 2$. Let us now set
\begin{displaymath}
\acal=\{A_Y:Y \in [\lambda]^n \} \subs [H]^t\,,
  \end{displaymath}
then clearly $|\acal| = \lambda$.

Since $|\br Y;n-1;\cap \br Z;n-1;|\le 1$ for distinct $Y,Z \in
[\lambda]^n$, we clearly have $|A_Y\cap A_Z|\le k$, hence $\acal$
is $(k+1)$-almost disjoint, i.e. $\acal$ is a
$(\lambda,t,k+1)$-system. Now, it remains to show that
$\ccf(\acal)>{\mu}$.

Assume that  $f:H \to \mu$ is given and define the map $$g:
[\lambda]^{n-1} \to\, {}^k\mu$$ by the stipulation
$g(B)(i)=f(\<B,i\>)\,$. By our partition relation hypothesis then
there is a $g$-homogeneous set $Y \in [\lambda]^n$. Consider an
arbitrary $\<B,i\>\in A_Y$. By (\ref{eq:B}) there is a $B' \ne B $
with $\<B',i\>\in A_Y$ as well, hence we have
$f(\<B,i\>)=g(B)(i)=g(B')(i)=f(\<B',i\>)$. Since $\<B,i\>$ was
arbitrary we obtain that $f$ is {\em not} a conflict free coloring
of $\acal$, completing the proof.
\end{proof}

We now list a number of easy but quite useful corollaries of
theorem \ref{tm:lb}.

\begin{corollary}\label{cor:wc}
If $\lambda = \omega$ or $\lambda$ is weakly compact then for any
$2 \le d \le t < \omega$ we have $\ccf(\lambda,t,d) = \lambda$.
\end{corollary}

\begin{proof}
To see this, let us first choose a natural number $n > 2$ such
that $t \le n\cdot (d-1)$. By our choice of $\lambda$, for every
$\mu < \lambda$ we have $\lambda \to (n)_{\mu^{d-1}}^{n-1}\,$, in
fact even $\lambda \to (\lambda)_{\mu^{d-1}}^{n-1}\,$. But then
theorem \ref{tm:lb} immediately yields $\ccf(\lambda,t,d) > \mu$,
hence as $\mu <  \lambda$ was arbitrary, $\,\ccf(\lambda,t,d) =
\lambda$.
\end{proof}

Since $\ccf(\omega_1,t,2) \ge \ccf(\omega,t,2)$, it immediately
follows from \ref{cor:wc} and the case $n = 2\,,\,k = 1$ of
theorem \ref{tm:ub} that $\ccf(\omega_1,t,2) = \omega$ whenever $3
\le t < \omega$. Similarly, comparing theorem \ref{tm:ubma} with
corollary \ref{cor:wc} we may conclude that $MA_\lambda(K)$
implies $\ccf(\lambda,t,d) = \omega$ whenever $d \ge 2$ and $t >
2(d-1)$.

An analogous argument as in the proof of corollary \ref{cor:wc},
using the case $n = 2$ of theorem \ref{tm:lb} and the trivial
partition relation $\lambda \to (2)^1_\kappa\,$ for all $\kappa <
\lambda$, yields the following result.

\begin{corollary}\label{cor:w1}
If $\lambda$ is infinite and $1 \le k  < \omega$,  then
$$\ccf(\lambda\,,2k\,,k+1) = \lambda.$$
\end{corollary}

On the basis of the conjecture that $\ccf(\lambda,\kappa,\mu)$ is
monotone decreasing in its second argument, it is natural to
expect from \ref{cor:w1} that we also have
$\ccf(\lambda\,,2k-1\,,k+1) = \lambda.$ We shall show below that
this is indeed true for ``most" values of $\lambda$, however the
full statement remains open in ZFC. We first give a somewhat
technical lemma.

\begin{lemma}\label{lm:3}
Let $\lambda$ be a cardinal that admits a coloring $f :
[\lambda]^2 \to \lambda$ of its pairs such that for any partition
$\pcal$ of $\lambda$ with $|\pcal| < \lambda$ there are $P \in
\pcal$ and $\{\alpha,\beta,\gamma\} \in [P]^3$ satisfying $f\{
\alpha,\beta\} = \gamma$. Then, for any $k > 1$, we have
$$\ccf(\lambda\,,2k-1\,,k+1) = \lambda\,.$$
\end{lemma}

\begin{proof}
$$\ical(f) = \big\{ \{\alpha,\beta\} \in [\lambda]^2 : f\{
\alpha,\beta\} \notin \{\alpha,\beta\} \big\}$$ naturally
decomposes into the following three parts:
$$\ical_0(f) = \big\{ \{\alpha,\beta\} \in [\lambda]^2 : f\{
\alpha,\beta\} < \alpha < \beta \},$$
$$\ical_1(f) = \big\{ \{\alpha,\beta\} \in [\lambda]^2 : \alpha < f\{
\alpha,\beta\} < \beta \},$$
$$\ical_2(f) = \big\{ \{\alpha,\beta\} \in [\lambda]^2 : \alpha< \beta< f\{
\alpha,\beta\}  \}.$$

We claim that our assumption on $f$ may be strengthened as
follows: There is a fixed $j < 3$ such that for any partition
$\pcal$ of $\lambda$ with $|\pcal| < \lambda$ there are $P \in
\pcal$ and $\{\alpha,\beta\} \in \ical_j(f) \cap [P]^2$ for which
$f\{\alpha,\beta\} \in P$.

Indeed, for every $j < 3$ let $g_j : [\lambda]^2 \to \lambda$ be
chosen in such a way that $g_j$ extends $f \upharpoonright
\ical_j(f)$. Then for one $j < 3\,$ the coloring $g_j$ together
with its index $j$ must satisfy the claim. Otherwise for every $j
< 3$ there is a partition $\pcal_j$ of $\lambda$ with $|\pcal_j| <
\lambda$ such that $g_j\{\alpha,\beta\} \notin P$ whenever
$\{\alpha,\beta\} \in \ical_j(g_j) \cap [P]^2$. But then $$\pcal =
\{P_1 \cap P_2 \cap P_3 : P_j \in \pcal_j,\,j < 3 \}$$ is a
partition of $\lambda$ with $|\pcal| < \lambda$ that cannot
satisfy our original assumption on $f$, a contradiction. So from
here on we assume that $f$ has the stronger property with $j$
fixed.

Take $\lambda$ many pairwise disjoint sets of size $k - 1\,$,
$\{H_\alpha : \alpha < \lambda\}$, and for each $\alpha < \lambda$
fix a member $h_\alpha \in H_\alpha$. For each $\{\alpha,\beta\}
\in \ical(f)$ let $$A_{\{\alpha,\beta\}} = H_\alpha \cup H_\beta
\cup \{h_{f\{\alpha,\beta\}}\}.$$ It is easy to check that then
$\acal = \{A_{\{\alpha,\beta\}} :  \{\alpha,\beta\} \in
\ical_j(f)\}$ is a $(\lambda\,,2k-1\,,k+1)$-system and we claim
that $\ccf(\acal) = \lambda$.

Indeed, consider any map $g : \cup \acal \to \kappa$ with $\kappa
< \lambda$. Then, by our assumption, there is a pair
$\{\alpha,\beta\} \in \ical_j(f)$ such that
$$g[H_\alpha] = g[H_\beta] = g[H_{f\{\alpha,\beta \}}]\,.$$
But clearly, every value taken by $g$ on $A_{\{\alpha,\beta \}}$
is taken at least twice, consequently $g$ is not a conflict free
coloring of $\acal$.
\end{proof}

Let us note that if $\lambda$ is regular and $f : [\lambda]^2 \to
\lambda$ establishes the negative partition relation $$\lambda
\nrightarrow [\lambda]^2_\lambda\,,$$ that is, $f[X] = \lambda$
for every $X \in [\lambda]^\lambda$, then $f$ trivially satisfies
the requirement of lemma \ref{lm:3} as well. Moreover, it is known
that $\lambda \nrightarrow [\lambda]^2_\lambda\,$ is valid
whenever $\lambda = \kappa^+$ for a regular cardinal $\kappa$, see
e.g. \cite{Sh}. Thus, we immediately obtain the following result.

\begin{corollary}\label{cor:odd}
If $\lambda$ is either a limit cardinal or the successor of a
regular cardinal and $1 < k < \omega$ then
$$\ccf(\lambda\,,2k-1\,,k+1) = \lambda\,.$$
\end{corollary}

\medskip

The following corollary of theorem \ref{tm:lb} uses, for $r = n- 2
> 0$, the well-known Erd\H os-Rado partition theorem
$$\beth_{r}(\kappa)^+ \to (\kappa^+)^{r+1}_\kappa\,.$$ Recall
that $\beth_r(\kappa)$ is defined by the recursion
$\beth_0(\kappa) = \kappa\,,\beth_{r+1}(\kappa)=
2^{\beth_r(\kappa)}$.

\begin{corollary}\label{cor:ER}
If $n \ge 3$ and $k < t \le n \cdot k$ then, for every $\kappa \ge
\omega$,
$$\ccf(\beth_{n-2}(\kappa)^+,t\,,k+1) > \kappa\,.$$
Consequently, if $\lambda$ is strong limit then for any $2 \le d
\le t < \omega$ we have $\ccf(\lambda,t,d) = \lambda$.
\end{corollary}
\begin{proof}
The first part, as mentioned, follows immediately from theorem
\ref{tm:lb} and the Erd\H os-Rado partition theorem. To see the
second, consider any $\kappa < \lambda$ and choose $n \ge 3$ such
that $t \le n\cdot (d-1)$. Then, by the first part, we have
$\ccf(\beth_{n-2}(\kappa)^+,t\,,k+1) > \kappa\,$, moreover
$\beth_{n-2}(\kappa)^+ < \lambda$ as $\lambda$ is strong limit,
hence $\ccf(\lambda,t,d) > \kappa$ as well. This completes the
proof as $\kappa < \lambda$ was arbitrary.
\end{proof}

Our next result yields a lower bound for $\ccf(\lambda,t,k+1)$ for
$t \le 2k$, like corollaries \ref{cor:w1} and \ref{cor:odd}. Of
course, if the statement of corollary \ref{cor:odd} turns out to
be valid for all $\lambda$, as we expect, then it becomes
superfluous.

\begin{theorem}\label{tm:lbch}
Assume that $\lambda$ and $\mu$ are infinite cardinals such that
$\lambda^{<\mu} = \lambda$, moreover  $0 < k < t \le 2k$ are
natural numbers. Then $$\ccf(\lambda,t,k+1) \ge \mu\,.$$
\end{theorem}

\begin{proof}
We are going to construct a $(\lambda,t,k+1)$-system $\acal \subs
[\lambda]^t$ that satisfies the following property
$\Phi(\lambda,\mu,k,t)\,$:

For every $Y \in [\lambda]^{<\mu}$ and for every {\em disjoint}
collection $B \subs [\lambda]^k$ with $|B| < \mu$ there is a set
$x \in [\lambda \setminus Y]^{t-k}$ such that $x \cup b \in \acal$
for each $b \in B$.

\smallskip

Before doing this, however, let us show that if $\acal$ satisfies
$\Phi(\lambda,\mu,k,t)\,$ then $\ccf(\acal) \ge \mu$. Indeed, let
$f : \lambda \to \nu$ be given for some $\nu < \mu$, where $\nu$
is infinite if $\mu > \omega$. Let us  put $S = \{\zeta < \nu :
|f^{-1}\{\zeta \}| \ge \omega \}$ if $\mu = \omega$ and $S =
\{\zeta < \nu : |f^{-1}\{\zeta \}| \ge \nu \}$ otherwise. We also
set $Y = \bigcup \{ f^{-1}\{\zeta \} : \zeta \in \nu \setminus S
\}$, clearly then $|Y| < \mu$. Next we consider the collection
$\scal = \{z \subs S : 0 <|z| \le t-k \}$, again we have $|\scal|
< \mu$. It is straight-forward to check that we may select for
each $z \in \scal$ a set $b_z \in [\lambda \setminus Y]^{k}$ so
that $f[b_z] = z$, moreover $B = \{b_z : z \in \scal \}$ is
disjoint.

By $\Phi(\lambda,\mu,k,t)\,$ there is some $x \in [\lambda
\setminus Y]^{t-k}$ such that $x \cup b_z \in \acal$ for each $z
\in \scal$. Now, $x \cap Y = \emptyset$ implies that $z = f[x] \in
\scal$, hence $x \cup b_z \in \acal$. But, as $x \cap b_z =
\emptyset$, the equality $f[x] = f[b_z]( = z)$ witnesses that $f$
is not a conflict free coloring of $\acal$, hence $\ccf(\acal) \ge
\mu$.

\smallskip

Now, we show how to construct $\acal$ satisfying
$\Phi(\lambda,\mu,k,t)\,$ by a transfinite recursion of length
$\lambda\,$. To start with, we fix a $\lambda$-type enumeration of
$[\lambda]^{<\mu} \times \bcal\,$:
$$[\lambda]^{<\mu} \times \bcal = \{\langle Y_\alpha,B_\alpha \rangle : \alpha < \lambda\}\,,$$
where $\bcal$ is the family of all disjoint collections $B \subs
[\lambda]^k$ with $|B| < \mu$. This is possible because
$\lambda^{<\mu} = \lambda$.

Next, assume that $\alpha < \lambda$ and for each $\beta < \alpha$
we have already constructed a $(k+1)$-almost disjoint family
$\acal_\beta \subs [\lambda]^t$ such that $|\acal_\beta| \le
\mu\cdot|\beta|$ if $\mu < \lambda$ and $|\acal_\beta| < \lambda$
if $\mu = \lambda$. We also assume that $\acal_\beta \subs
\acal_\gamma$ whenever $\beta < \gamma < \alpha$.

Now, if $\alpha$ is limit then we simply put $\acal_\alpha =
\cup_{\beta < \alpha}\acal_\beta$. It is easy to see that then all
our inductive hypotheses remain valid. This is obvious if $\mu <
\lambda$, and if $\mu = \lambda$ then it follows because $\lambda$
is regular by the assumption $\lambda^{<\lambda} = \lambda\,$.

If, on the other hand, $\alpha = \beta + 1\,$ then we consider the
pair $\langle Y_\beta\,,B_\beta \rangle\,$ and choose a set $x \in
[\lambda]^{t-k}$ that is disjoint from $\,\,\bigcup \acal_\beta
\cup \bigcup B_\beta \cup Y_\beta\,$. Then we put $$\acal_\alpha =
\acal_{\beta+1} = \acal_\beta \cup \{ b \cup x : b \in B_\beta
\}\,.$$ Again, it is obvious that our inductive hypotheses remain
valid.

Finally, if the transfinite recursion is completed, then we set
$$\acal = \bigcup \{ \acal_\alpha : \alpha < \lambda \}\,.$$
It is obvious from our construction that $\acal \subs [\lambda]^t$
is a $(\lambda,t,k+1)$-system that satisfies property
$\Phi(\lambda,\mu,k,t)\,$ and hence $\ccf(\acal) \ge \mu\,$.
\end{proof}

\begin{corollary}\label{cor:ch}
Let $k$ and $t$ be integers with $1 \le k < t \le 2k$. If
$\kappa^+ = 2^\kappa$ then $\ccf(\kappa^+,t,k+1) = \kappa^+ $.
\end{corollary}

In particular, as we promised, CH implies $\ccf(\omega_1,t,k+1) =
\omega_1$ for any such $k$ and $t$. Actually, our previous results
enable us to give, under the assumption of GCH, a complete and
rather attractive description of the behavior of
$\,\ccf(\lambda,t,k+1)\,$ for all $\lambda \ge \omega > t > k \ge
1$.

\begin{theorem}\label{tm:gch}
Assume GCH and let $\kappa$ be any limit cardinal or $\kappa =
\omega$, moreover fix the natural number $k \ge 1$. Then for any
$n<\omega$ we have
\[
\ccf(\kappa^{+n},t,k+1) =\left\{
\begin{array}{lll}
\kappa^{+(n+1-i)}&\mbox{ if $\,\,i\cdot k < t \le (i+1)\cdot k\,,$}\\&
\makebox[75pt]{}\mbox{$i = 1,...,n$;}\\
{}\\
\kappa& \mbox{ if $\,\,(n+1)\cdot k < t\,$}.
\end{array}
\right.
\]

\end{theorem}

\medskip

\begin{proof}
Let us note first that by the second part of corollary
\ref{cor:ER} and by corollary \ref{cor:wc} we have
$\ccf(\kappa,t,k+1) = \kappa$ for all $0 < k < t < \omega$ which
shows that our claim holds for $n = 0$. So, from here on we fix $n
\ge 1$.

Let us assume now that $k < t \le 2k$. In this case we may apply
corollary \ref{cor:ch} to $\kappa^{+n} = 2 ^{(\kappa^{+n-1})}$ and
conclude that
$$\ccf(\kappa^{+n},t,k+1) = \kappa^{+n} = \kappa^{+(n+1-1)}.$$

Next, consider the case $i\cdot k < t \le (i+1)\cdot k\,$ with $2
\le i \le n$. Then from $i\cdot k < t$, applying theorem
\ref{tm:ub} to the cardinal $\kappa^{+(n+1-i)}$ and the number
$i$, we obtain $\ccf(\kappa^{+n},t,k+1) \le \kappa^{+(n+1-i)}$.
>From $\,t \le (i+1)\cdot k\,$, on the other hand, applying
corollary \ref{cor:ER} to the number $i+1 \ge 3$ and the cardinal
$\kappa^{+(n-i)}$ we obtain the converse inequality
$\ccf(\kappa^{+n},t,k+1) \ge \kappa^{+(n+1-i)}$.

Finally, assume that $t > (n+1)\cdot k$. Then from theorem
\ref{tm:ub}, applied with the number $n+1$, we conclude
$\ccf(\kappa^{+n},t,k+1) \le \kappa$. But then we must have
$\ccf(\kappa^{+n},t,k+1) = \kappa$ because already
$\ccf(\kappa,t,k+1) = \kappa$.

This concludes the proof because we have checked all the cases.
\end{proof}

It is immediate from theorem \ref{tm:gch} that, in accordance with
our earlier conjecture,  $\ccf(\lambda,t,d)$ is a monotone
decreasing function of $t < \omega$ for fixed $\lambda$ and $d$,
at least if GCH holds.

\smallskip

\begin{problem}
Is $\,\ccf(\lambda,2k-1,k+1) = \lambda\,$ provable in ZFC for all
$\lambda \ge \omega$ and $1 < k < \omega$?
\end{problem}

\bigskip

\bigskip

\begin{center}
\sc Part II. The case $\lambda \ge \kappa \ge \omega > \mu$
\end{center}

\bigskip

\section{$\omega$ colors suffice}
\label{sc:dfirst}

It follows from theorem \ref{tm:ub} that if $\lambda <
\aleph_\omega$ then, for fixed $d < \omega$, we have
$\ccf(\lambda,t,d) \le \omega$ provided that $t < \omega$ is large
enough. The result we prove in this section shows that if we
replace $t$ with any infinite cardinal $\kappa$ then
$\ccf(\lambda,\kappa,d) \le \omega$ holds for {\em all} $\lambda
\ge \kappa$.

\begin{theorem}\label{tm:korlatos}
For any ${\lambda}\ge \kappa \ge \omega$ and $d < \omega$ we have
$\ccf(\lambda,\kappa,d) \le \omega$, in fact even the stronger
relation $[\lambda,\kappa,d] \Rightarrow \omega$.
\end{theorem}

\begin{proof}[First proof of Theorem  \ref{tm:korlatos}]
We prove $[\lambda,\kappa,d] \Rightarrow \omega$ by transfinite
induction on $\kappa$ and $\lambda$\,: Assuming
$[\kappa',\kappa',d] \Rightarrow \omega$ and $[\lambda',\kappa,d]
\Rightarrow \omega$ for all ${\omega}\le {\kappa}'<{\kappa}$ and
${\kappa}\le {\lambda}'< {\lambda}\,$, we deduce
$[\lambda,\kappa,d] \Rightarrow \omega$.

\medskip
\noindent{\bf Case 1:} ${\lambda}={\kappa}={\omega}$.\\
Let $\acal=\{A_n:n<{\omega}\} \subs [\omega]^\omega$ be $d$-almost
disjoint (actually, $\omega$-almost disjoint would suffice) and
construct $c:{\omega}\to {\omega}$ in such a way that
$c\restriction A_n\setm \cup\{A_m:m<n\}$ is a bijection with range
${\omega}$ for each $n<{\omega}$. Thus $\omega \setminus I_c(A_n)
\subs c[A_n\cap \cup\{A_m:m<n\}]$ is finite for all $A_n \in
\acal$, and we are done.

\medskip

\noindent{\bf Case 2:} ${\lambda}={\kappa}>{\omega}$.\\
Let $\acal\subs \br {\kappa};{\kappa};$
be  $d$-almost disjoint and
 $\<N_{\alpha}:{\alpha}<{\kappa}\>$ be a ${\kappa}$-chain
of elementary submodels with $\acal\in N_1$. For
${\alpha}<{\kappa}$ let $\kappa_\alpha = |N_{\alpha+1}|$,
$B_{\alpha}=\cup(\acal\cap N_{\alpha})$ and
$Y_{\alpha}=N_{{\alpha}+1} \cap ({\kappa}\setm (B_{\alpha}\cup
N_{\alpha}))$.

If $A\in \acal\cap N_{{\alpha}+1}\setm N_{\alpha}$ then
\begin{displaymath}
 |A\cap B_{\alpha}|\le \sum\{|A\cap A'|:A'\in\acal\cap N_{\alpha}\}\le
|N_{\alpha}|\cdot d<{\kappa},
\end{displaymath}
and so  $ |{\kappa}\setm (B_{\alpha}\cup N_{\alpha})| = |A\setm
(B_{\alpha}\cup N_{\alpha})|={\kappa}$. But $A\setm
(B_{\alpha}\cup N_{\alpha})\in N_{{\alpha}+1}$ and $\kappa_\alpha
\subs N_{{\alpha}+1}$ imply
\begin{displaymath}
|Y_\alpha| = |A\cap Y_{\alpha}|=|N_{{\alpha}+1} \cap (A\setm
(B_{\alpha}\cup N_{\alpha}))|=\kappa_\alpha,
\end{displaymath}
consequently
\begin{displaymath}
\acal_{\alpha}=\{A \cap Y_{\alpha}:A\in \acal \cap
N_{{\alpha}+1}\setm N_{\alpha}\} \subs [Y_\alpha]^{\kappa_\alpha},
\end{displaymath}
and $\acal_{\alpha}$ is clearly  $d$-almost disjoint. By the
inductive assumption $[\kappa_\alpha, \kappa_\alpha,d] \Rightarrow
\omega$, there is a function
 $c_{\alpha}: Y_{\alpha}\to {\omega}$
such that ${\omega}\setm I_{c_{\alpha}}(A')$ is finite for all
$A'\in\acal_{\alpha}$.

Let $c'=\cup \{c_{\alpha}:{\alpha}<{\kappa}\}$ and consider the
function $c\supset c'$ which maps ${\lambda}$ into ${\omega}$ in
such a way that $c[{\lambda}\setm \dom(c')]\subs \{0\}$. Now, let
$A\in \acal$ and ${\alpha}<{\kappa}$ be such that $A\in
N_{{\alpha}+1}\setm N_{\alpha}$. Then $A'=A\cap Y_{\alpha}\in
\acal_{\alpha}$, so ${\omega}\setm I_{c_{\alpha}}(A')$ is finite.
But we also have
\begin{equation}\label{eq:bullet1x}
|A \cap \dom (c'\setm c_{\alpha})|<d.
\end{equation}
Indeed, if ${\alpha}<{\beta} < \kappa\,$ then $\,A \cap Y_{\beta}
=\empt$, while $A\notin N_{\alpha}$ implies $|A\cap
N_{\alpha}|<d$, and hence $|A\cap
\cup\{Y_{\gamma}:{\gamma}<{\alpha}\}|<d$ as well. Since
\begin{equation}
I_{c_{\alpha}}(A')\setm I_c(A)
\subset c'[A\setm \dom c_{\alpha}]\cup \{0\},
\end{equation}
it follows that
 ${\omega}\setm I_c(A)$
is finite, and we are done.

\medskip
\noindent{\bf Case 3:} ${\lambda}>{\kappa}$.

Let $\acal \subs \br {\lambda};{\kappa};$ be  $d$-almost disjoint
and $\<N_{\alpha}:{\alpha}<{\lambda}\>$ be a ${\lambda}$-chain of
elementary  submodels with $({\kappa}+1)\cup\{\acal\}\subs N_1$.
For each ${\alpha}<{\lambda}$ let $Y_{\alpha}={\lambda}\cap
N_{{\alpha}+1}\setm N_{\alpha}$, then $\kappa \le |Y_\alpha| =
|N_{\alpha+1}| < \lambda$.

For any $A \in \acal\cap N_{{\alpha}+1}\setm N_{\alpha}$ we have
$|A\cap N_{\alpha}|<d$ and $A\subs N_{{\alpha}+1}$, hence
\begin{displaymath}
\acal_{\alpha}=\{A \setm  N_{\alpha}:A\in \acal\cap
N_{{\alpha}+1}\setm N_{\alpha}\} \subs \br Y_{\alpha};{\kappa};,
\end{displaymath}
and $\acal_{\alpha}$ is $d$-almost disjoint. Now, we may argue
inductively, exactly as in Case 2, to obtain a map $c : \lambda
\to \omega$ such that $\omega \setm I_c(A)$ is finite for each $A
\in \acal$.
\end{proof}

\begin{remark}
P.~Komj\'ath pointed out to us an easy proof of Theorem
\ref{tm:korlatos} for the case $\kappa = \omega$. His proof relied
on a result of his proved in \cite{KO2} claiming that every
$(\lambda,\omega,d)$-system $\acal$ is {\em essentially disjoint},
i.e. one can omit a finite set $F(A)$ from each element $A$ of
$\acal$ in such a way that the sets $A \setm F(A) $ are pairwise
disjoint. By taking a bijection between $A \setm F(A) $ and
$\omega$ for each $A \in \acal$, and then coloring the rest
arbitrarily, we get an appropriate ${\omega}$-coloring. Based on
this observation, and a result of Erd\H os and Hajnal, we shall
give a short alternative proof of theorem \ref{tm:korlatos}.
\end{remark}

We recall from \cite{HJS1} and \cite{HJS2} that a set $X$ is
called a $\tau$-transversal of a family $\,\acal\,$ if $\,0 < |X
\cap A| < \tau\,$ for all $A \in \acal$. Moreover, the symbol $\MM
{\lambda}{\kappa}{\mu}\to \BB {\tau}$ is used there to denote the
statement that every $(\lambda,\kappa,\mu)$-system has a
$\tau$-transversal.

For us it will be useful to introduce the following variation on
this concept: We say that $X$ is a $\tau$-witness for $\acal$ iff
$\,|X \cap A| = \tau\,$ for all $A \in \acal$. Clearly, any
$\tau$-witness is a $\tau^+$-transversal. It is easy to see that
if $\kappa \ge \tau \ge \omega$ then $\MM
{\lambda}{\kappa}{\mu}\to \BB {\tau^+}$ holds iff every
$(\lambda,\kappa,\mu)$-system has a $\tau$-witness.

\begin{definition}
A  $(\lambda,\kappa)$-family $\mc A$ is called {\em essentially
disjoint} (ED, in short) iff for each $A\in \mc A$ there is a set
$F(A)\in \br A;<{\kappa};$ such that the family $\{A\setm
F(A):A\in \mc A\}$ is disjoint.  $\MM {\lambda}{\kappa}{\mu}\to
{\bf ED}$ denotes the statement that every
$(\lambda,\kappa,\mu)$-system is ED.
\end{definition}

\begin{proposition}\label{lm:reduction}
Assume $ {\mu}\le \tau \le {\kappa}\le {\lambda}$ and $\tau \ge
\omega$. Then $$\MM {\lambda}{\kappa}{\mu}\to \BB {{\tau}^+}
\mbox{ and }\, \MM {\lambda}{\tau}{\mu}\to {\bf ED}$$ together
imply $[\lambda,\kappa,\mu] \Rightarrow \tau\,.$
\end{proposition}

\begin{proof}
Let $\mc A\subs \br {\lambda};{\kappa};$ be a
$(\lambda,\kappa,{\mu})$-system. Since $\MM
{\lambda}{\kappa}{\mu}\to \BB {{\tau}^+}$ there is a
$\tau$-witness $X$ for $\acal$. Then
\begin{displaymath}
\acal'=\mc A\restriction X=\{A\cap X:A\in \mc A\}\subs \br
{\lambda};{\tau};
\end{displaymath}
is a $(\lambda,\tau,\mu)$-system. Applying $\MM
{\lambda}{\tau}{\mu}\to {\bf ED}$ for $\acal'$ there is a function
$F:\mc A'\to \br {\lambda};<{\tau};$ such that the family
$\{A'\setm F(A'):A'\in \mc A'\}$ is disjoint.

Let $c:{\lambda}\to {\tau}$ be a function such that
\begin{enumerate}[(i)]
\item $c[{\lambda}\setm X]=\{0\}$,
\item $c\restriction A'\setm F(A')$ is a bijection
between $A'\setm F(A')$ and ${\tau}$ for each $A'\in \mc A'$.
\end{enumerate}
Then for each $A\in \mc A$,
\begin{displaymath}
{\mu}\setm I_c(A)\subs \{0\}\cup c[F(A\cap X)]\in \br
{\tau};<{\tau};,
\end{displaymath}
i.e. $c$ witnesses $\Carr {\kappa}{\lambda}{\rho}{\mu}$
\end{proof}

\begin{proof}[Second proof of theorem \ref{tm:korlatos}]
In \cite[Theorem 8(b)]{EH3} Erd\H os and Hajnal proved that
\begin{equation}\label{eh1x}
\text{$\MM {\lambda}{\kappa}d\to \BB {\omega}$ for $d<{\omega}\le
{\kappa}\le {\lambda}$.}
\end{equation}
Moreover, in \cite[Theorem 2]{KO2}, Komj\'ath proved
\begin{equation}\label{ko1}
\text{$\MM {\lambda}{\omega}d\to {\bf ED}$ for $d<{\omega}\le
{\lambda}$,}
\end{equation}
By proposition \ref{lm:reduction}, (\ref{eh1x}) and (\ref{ko1})
imply $\Carr {\kappa}{\lambda}d{\omega}$. (Actually, instead of
(\ref{eh1x}), $\MM {\lambda}{\kappa}d\to \BB {\oo}$ would be
enough.)
\end{proof}

As a matter of fact, the theorem of Erd\H os and Hajnal,
\cite[Theorem 8(b)]{EH3} that we stated and used above can be
proved with the method of elementary chains as presented in the
first proof of theorem \ref{tm:korlatos}. Moreover, we should
point out that all the results mentioned in this section can also
be deduced from the very general, and therefore rather technical,
main theorem 1.6 of \cite{HJS1}.

\bigskip

\bigskip

\section{A finite upper bound for $w\ccf(\kappa^{+m},\kappa,d)$}

We have seen in the previous section that $\ccf(\lambda,\kappa,d)$
is countable whenever $\lambda\ge \kappa \ge \omega > d$. The aim
of this section is to show that if $\lambda$ is ``not much bigger
than" $\kappa$, namely it is a finite successor of $\kappa$, then
$\ccf(\lambda,\kappa,d)$ is even finite. This is immediate from
the following theorem that is formulated in terms of the weak
conflict free chromatic number.

\begin{theorem}\label{tm:egeszresz}
If ${\kappa}$ is infinite, $d > 0$ and $m$ are natural numbers
then $$w\ccf(\kappa^{+m},\kappa,d) \le {\left\lfloor
{\frac{(m+1)(d-1)+1}2}   \right\rfloor+1}\,,$$ or with our
alternative arrow notation:
\begin{equation}
\label{dag}
\carrw {\kappa}{{\kappa}^{+m}}d
{\left\lfloor {\frac{(m+1)(d-1)+1}2}   \right\rfloor+1}.
\end{equation}
\end{theorem}

We shall actually prove a stronger result than theorem
\ref{tm:egeszresz}. This involves a refined version of our weak
arrow relation whose definition is given next. In this we shall
use $\fcal(A,B)$ to denote the set of all partial functions from
$A$ to $B$.

\begin{definition}\label{def:gen_princ}
Let $ {\lambda} \ge \kappa \ge \omega$ and $d,k,x\in {\omega}$.
Then $$ \Induw {\lambda}{\kappa}dkx$$ abbreviates the following
statement: If $C \subs \lambda$ and $\acal\subs \br
{\lambda};{\kappa};$ is any $d$-almost disjoint system satisfying
$|A\cap C|\le k$ for each $A\in \acal$, then for every partial
function $f \in \fcal(C,x)$ there is a weak conflict free coloring
$g \in \fcal(\lambda,x)$ of $\acal$ such that $g\upharpoonright C
= f$. Note that the last equality is equivalent to $g \supset f$
and $C \cap \dom(g) = \dom(f)$.

For later use we also define the (strict) relation $\Indu
{\lambda}{\kappa}dkx$ as follows: For any $d$-almost disjoint
$\acal\subs \br {\lambda};{\kappa};$ and $f \in \fcal(\lambda,x)$
satisfying $|A\cap \dom(f)|\le k$ for each $A\in \acal$, there is
a conflict free coloring $g : \lambda \to x$ of $\acal$ with $g
\supset f$.
\end{definition}

The main result of this section may be then formulated as follows.
(Note that theorem \ref{tm:egeszresz} is an immediate corollary of
the particular case $k = 0$ of theorem \ref{tm:ind}.)

\begin{theorem}\label{tm:ind}
Let  ${\kappa}$ be an infinite cardinal and $m,d,k $ be natural
numbers with $d>0$. Then
\begin{displaymath}
\induw {\kappa}dmk{\left\lfloor {\frac{(m+1)(d-1)+k+1}{2}}   \right\rfloor+1}.
\end{displaymath}
\end{theorem}

The proof of theorem \ref{tm:ind} will be carried out by induction
on $m$, using theorems \ref{tm:ind0} and \ref{tm:ind_new} below.

\begin{theorem}\label{tm:ind0}
Let  ${\kappa}$ be an infinite cardinal, moreover $\,d$ and $\,x $
be natural numbers with $2x > d$. Then $$\induw
{\kappa}d0{2x-d-1}x\,.$$
\end{theorem}

\begin{proof}[Proof of Theorem \ref{tm:ind0}]
Let us write $k=2x-d-1$ and assume that a set $C \subs \kappa$, a
$d$-almost disjoint system $\acal\subs \br {\kappa};{\kappa};$,
and a partial function $f \in \fcal(C,x)$ are given  such that
$|A\cap C|\le k$ for each $A\in \acal$. We may clearly assume that
$|\acal| = \kappa$, and hence may fix a one-to-one $\kappa$-type
enumeration $\{A_\eta : \eta < \kappa\}$ of $\acal$.

By transfinite induction we shall define $f_{\eta} \in
\fcal(\kappa,x)$ for ${\eta}\le {\kappa}$ such that the inductive
conditions (i) - (iv) below be valid.
\begin{enumerate}[(i)]
\item
$f_{\eta} \supset f_\zeta \supset f$ for $\eta > \zeta\,$,
\item
$\,C \cap\dom(f_{\eta}\setm f) = \emptyset$ and $|f_{\eta}\setm
f|\le |{\eta}|\,$,
\item $\forall\, {\zeta}<{\eta}$ $\exists\, i<x$
$|A_{\zeta}\cap f_{\eta}^{-1}\{i\}|=1\,$,
\item if ${\gamma}\ge {\eta}$ then
$|A_{\gamma}\cap \dom(f_{\eta}\setm f)|\le d$.
\end{enumerate}
\noindent{\bf Case 1.} {\em ${\eta}=0$.}\\
Put $f_0=f$, then (i) - (iv) hold trivially.

\medskip

\noindent{\bf Case 2.}  ${\eta}$ is a limit ordinal.\\
Put $f_{\eta}=\cup\{f_{\zeta}:{\zeta}<{\eta}\}$. It is again easy
to check that the validity of conditions (i) - (iv) will be
preserved. In particular, (iii) is preserved because, as $x$ is
finite, for each $\zeta < \eta$ there are cofinally many $\xi <
\eta$ satisfying $|A_{\zeta}\cap f_{\xi}^{-1}\{i\}|=1\,$ with the
same $i < x$.
\medskip

\noindent{\bf Case 3.} {\em ${\eta}={\zeta}+1$.}\\
Then we have
\begin{multline}\notag
|A_{\zeta}\cap \dom(f_{\zeta})| = |A_{\zeta}\cap \dom(f)|+
|A_{\zeta}\cap (\dom(f_{\zeta}\setm f))|\le k+d\,,
\end{multline}
consequently,  $2x>2x-1=k+d$ implies that there is $i<x$ such that
$|A_{\zeta}\cap f_{\zeta}^{-1}\{i\}|\le 1$. If there is an $i<x$
such that $|A_{\zeta}\cap f_{\zeta}^{-1}\{i\}|= 1$ then the choice
$f_{\eta}=f_{\zeta}$ clearly works.

Otherwise we may fix $j<x$ with $A_{\zeta}\cap
f_{\zeta}^{-1}\{j\}=\empt$. Let us then put
\begin{displaymath}
\acal_{\zeta}=\{A_{\xi}:{\xi}<{\zeta}\}\cup\{
A_{\gamma}:{\zeta}<{\gamma}\land |A_\gamma \cap
\dom(f_{\zeta}\setm f)| = d\}\,.
\end{displaymath}
Using $|\dom(f_{\zeta}\setm f)|\le |{\zeta}| < \kappa$ and that
$\acal$ is $d$-almost disjoint we get $|\acal_{\zeta}| <{\kappa}$,
moreover we also have $|C \cap A_\zeta| \le k$. Thus we can pick
\begin{displaymath}
\xi_{\zeta}\in A_{\zeta}\setm \big(\cup\acal_{\zeta} \cup C \big)
\end{displaymath}
and put $$f_{\eta} = f_{\zeta+1} =f_{\zeta} \cup \{\langle
\xi_\zeta,j \rangle\}.$$

\smallskip

Then $f_{\eta}$ clearly satisfies (i) and (ii). If ${\xi}<{\zeta}$
then, by our construction, $f_{\eta}\restriction
A_{\xi}=f_{\zeta}\restriction A_{\xi}$, hence, as (iii) is
satisfied by $f_\zeta$, there is $i<x$ such that
$|f_{\eta}^{-1}\{i\}\cap A_{\xi}|=1$. Moreover,
$f_{\eta}^{-1}\{j\}\cap A_{\zeta}=\{\xi_{\zeta}\}$, so (iii) is
satisfied by $f_{\eta}$ as well.

Finally, to show that $f_\eta$ satisfies (iv), consider any
$\gamma \ge {\eta}$. If we have $|A_{\gamma}\cap\,
\dom(f_{\zeta}\setm f)|< d$ then $|A_{\gamma}\cap\,
\dom(f_{\eta}\setm f )|\le d$ holds trivially, because
$|\dom(f_{\eta} \setm f_{\zeta}|\le 1$. If, on the other hand,
$\,|A_{\gamma}\cap \dom(f_{\zeta}\setm f)| = d$ then
$A_{\gamma}\in \acal_{\zeta}$ and so $\xi_{\zeta}\notin
A_{\gamma}$. Thus, in this case, $|A_{\gamma}\cap\,
\dom(f_{\eta}\setm f)|= |A_{\gamma}\cap\, \dom(f_{\zeta}\setm
f)|=d$; in any case $f_\eta$ satisfies (iv).

Obviously, then $f_{\kappa} \in \fcal(\lambda,x)$ is a weak
conflict free coloring of $\acal$ that satisfies $f_\kappa
\upharpoonright C = f$, completing the proof.
\end{proof}

Next we prove a stepping up result for the first parameter of our
new arrow relations. The proof of this will reveal why we chose to
introduce this new relation.

\begin{theorem}\label{tm:ind_new}
Let $\lambda \ge {\kappa} \ge \omega$ and $\,d,k,x\in {\omega}$
with $d > 0$. Then
\begin{enumerate}[$(1)$]
\item $[\lambda,\kappa,d,k+d-1] \to_w x$ implies $[\lambda^+,\kappa,d,k] \to_w x$,
\smallskip
\item  $[\lambda,\kappa,d,k+d-1] \to x$ implies $[\lambda^+,\kappa,d,k] \to x$.
\end{enumerate}
\end{theorem}

\begin{proof}[Proof of Theorem \ref{tm:ind_new}]
(1). Assume that $C \subs \lambda^+$ and the $d$-almost disjoint
system $\acal\subs \br {\lambda^+};{\kappa};$ are such that $|A
\cap C| \le k$ for any $A \in \acal$, moreover $f \in \fcal(C,x)$.
Let $\<N_{\nu}:{\nu}<\lambda^+\>$ be a $\lambda^+$-chain of
elementary submodels such that $\lambda^+,\acal,C,f \in N_1$ and
$\lambda\subs N_1$.

By transfinite induction we shall define $g_{\eta}\in
\fcal(\lambda^+,x)$ for all ${\eta}< \lambda^+$ satisfying the
following inductive hypotheses.
\begin{enumerate}[(i)]
\item $\dom(g_{\eta})\subs N_{{\eta}}$ and $g_\zeta \subset g_\eta$ for $\zeta < \eta$,
\item $g_\eta \upharpoonright C \cap N_\eta = f\restriction C \cap N_{\eta}$,
\item
$g_{\eta}$ is a weak conflict free coloring of $\acal\cap
N_{\eta}$.
\end{enumerate}

\noindent{\bf Case 1.} {\em ${\eta}=0$.}\\
We have to put $g_0 = \emptyset\,$ because $N_0 = \emptyset$. This
works trivially for the same reason.

\medskip

\noindent{\bf Case 2.} {\em ${\eta}$ is limit.}\\
Then we put $g_{\eta}=\cup\{g_{\zeta}:{\zeta}<{\eta}\}$. Now, (i)
and (ii) follow immediately from
$N_{\eta}=\cup\{N_{\xi}:{\xi}<{\eta}\}$. To check (iii), pick
$A\in \acal\cap N_{\eta}$. There is a $\zeta < \eta$ with $A \in
N_\zeta$ and so for every $\nu \in \eta \setm \zeta$ there is
$i_\nu < x$ with $|A\cap  g_{\nu}^{-1}\{i_\nu\}|=1$. As $x$ is
finite, we have an $i < x$ such that $i_\nu = i$ for cofinally
many $\nu \in \eta$, hence $|A\cap  g_{\eta}^{-1}\{i\}|=1$.

\medskip

\noindent{\bf Case 3.} {\em ${\eta}={\zeta}+1$.}\\ Let us put
$C_\eta = (C \cap N_\eta) \cup N_{\zeta}$ and $f_{\eta}=
(f\restriction N_{\eta}) \cup g_{\zeta} \in \fcal(C_\eta,x)$. Then
for all $A\in \acal \cap (N_\eta \setm N_{\zeta}) \subs [\lambda^+
\cap N_\eta]^{\kappa}$ we have
$$|A\cap C_{\eta}|\le  |A\cap C|+|A\cap N_{\zeta}|\le
k+(d-1)\,.$$

But $|\lambda^+ \cap N_\eta| = \lambda$, hence we can apply
$[\lambda,\kappa,d,k+d-1] \to_w x$  to $C_\eta$, $\acal \cap
(N_\eta \setm N_{\zeta})$, and $f_{\eta}$ to find a weak conflict
free coloring $g_{\eta}$ of $\acal\cap (N_{\eta} \setm N_\zeta)$
such that $\dom(g_{\eta})\subs \lambda^+ \cap N_{\eta}$ and
$$g_{\eta}\upharpoonright C_\eta = f_{\eta}\upharpoonright
C_\eta\,.$$ In particular, then $g_\zeta \subs g_\eta$ and since
for every $A \in \acal \cap N_\zeta$ we have $A \subs N_\zeta
\subs C_\eta$ we obtain that $g_\eta$ is a weak conflict free
coloring $g_{\eta}$ of $\acal \cap N_\eta$. Finally,  $C \cap
N_\eta \subs C_\eta$ implies $g_\eta \upharpoonright C \cap N_\eta
= f\restriction C \cap N_{\eta}$, which shows that $g_\eta$
satisfies all three inductive hypotheses and thus completes the
inductive construction.

It is now obvious that the function $g= \bigcup_{\eta <
\lambda^+}g_\eta$ is a weak conflict free coloring of $\acal$ and
satisfies $g \upharpoonright C = f$, which completes the proof of
(1).

\medskip
\noindent (2) can be proved in a completely similar, but even
simpler, manner.
\end{proof}

\begin{proof}[Proof of Theorem \ref{tm:ind}]
To start with, in the case $m = 0$, we have to show
$$[\kappa,\kappa,d,k] \to_w \left\lfloor \frac{k+d}2\right\rfloor +1   $$
for all natural numbers $k$ and $d$. To see this, put $x =
\left\lfloor \frac{k+d}2\right\rfloor +1$ and note that we have
$2x \ge k+d+1$, hence $2x-d-1 \ge k$. But then, applying theorem
\ref{tm:ind0}, we can conclude $\induw {\kappa}d0{2x-d-1}x\,$ and
hence $\induw {\kappa}d0{k}x\,$ as well.

Now, assume that $m > 0$ and theorem \ref{tm:ind} has been
verified for $m-1$, i.e.

\begin{displaymath} \induw
{\kappa}d{m-1}k{\left\lfloor {\frac{m(d-1)+k+1}{2}}
\right\rfloor+1}
\end{displaymath}
holds for all $d$ and $k$. Applying the stepping up theorem
\ref{tm:ind_new} to this formula with $k$ replaced by $k+d-1$ (and
$\lambda = \kappa^{+m-1}$) we obtain
\begin{displaymath}
\induw {\kappa}dmk{\left\lfloor {\frac{(m+1)(d-1)+k+1}{2}}
\right\rfloor+1},
\end{displaymath}
completing the induction step from $m-1$ to $m$.
\end{proof}

\section{A lower bound for $w\ccf(\beth_m(\kappa),\kappa,d)$}
\label{sc:dlast}

Now we know that $w\ccf(\kappa^{+m},\kappa,d)$ is finite, hence it
is natural to attempt to find its exact value. The aim of this
section is to execute this attempt, at least under GCH and for $d
= 2$ or $d$ odd. The case $m = 0$ is relatively easy to deal with,
using the following lemma.

\begin{lemma}\label{lm:0*}
Fix a cardinal $\kappa \ge \omega$ and a natural number $t > 0$.
We have a procedure that assigns to any
$(\kappa,\kappa,2t)$-system $\fcal$ another
$(\kappa,\kappa,2t)$-system $\fcal^*$ in such a way that
$w\ccf(\fcal^*) > t$ holds whenever $w\ccf(\fcal) \ge t$.
\end{lemma}

\begin{proof}
Given any $(\kappa,\kappa,2t)$-system $\fcal$, let us first choose
pairwise disjoint sets
$\{A_n:n<2t\}\cup\{B_{\nu}:{\nu}<{\kappa}\}\subs \br
{\kappa};{\kappa};$. For each $n<2t$ let $\acal_n\subs \br
A_n;{\kappa};$ be an isomorphic copy of $\fcal$. Let
$\{X_{\nu}:{\nu}<{\kappa}\}$ be a one-one enumeration of the
family
\begin{displaymath}
  \{X : |X| = 2t \mbox{ and } |X\cap A_n|= 1 \text{ for each } n<2t\}
\end{displaymath}
of all transversals of  $\{A_n:n<2t\}$. Write $C_{\nu}=B_{\nu}\cup
X_{\nu}$ and let
\begin{displaymath}
\fcal^*=\bigcup\{\acal_n:n<{2t}\}\cup \{C_{\nu}:{\nu}<{\kappa}\}.
\end{displaymath}
Then $\fcal^*$ is $2t$-almost disjoint, because
\begin{itemize}
\item $|C_{\nu}\cap C_{\mu}|=|X_{\nu}\cap X_{\mu}| < 2t$
 for ${\nu} \ne {\mu}$,
\item $|C_{\nu}\cap A|\le 1$ for $A\in \bigcup_{n<{2t}}\acal_n$.
\end{itemize}
Now, assume that $w\ccf(\fcal) \ge t$ and, contrary to our claim,
$h$ is a weak conflict free coloring of $\fcal^*$ with color set
$t$. Then $w\ccf(\acal_n)\ge t$ implies that $h[A_n] = t$ for each
$n < 2t$, hence for each $i < t$ there are $x_i\in A_{2i}$ and
$y_i\in A_{2i+1}$ such that $h(x_i)=h(y_i)=i$.

There is a ${\nu}<{\kappa}$ with   $X_{\nu}=\{x_i, y_i:i<t\}$. But
then $|h^{-1}\{i\}\cap C_{\nu}|\ge  2$ for each $i<t$,  and $h$ is
not a  weak conflict free coloring of $\fcal^*$, a contradiction.
So, indeed, we have $w\ccf(\fcal^*)> t$.

\end{proof}

\begin{theorem}\label{tm:omegaw}
For any cardinal ${\kappa} \ge \omega$ and integer $d \ge 2$ we
have
  \begin{displaymath}\tag{$*_d$}
\crcw {\kappa}{\kappa}d={\left\lfloor {\frac{d}2}   \right\rfloor+1}.
    \end{displaymath}
\end{theorem}

\begin{proof}
We shall prove, by  induction on $1\le s<{\omega}$, that
\begin{displaymath}\tag{$\circ_s$}
\crcw {\kappa}{\kappa}{2s}\ge s+1.
\end{displaymath}
Then, also applying theorem \ref{tm:egeszresz}, we have
\begin{displaymath}
s+1 \le \crcw {\kappa}{\kappa}{2s}\le\crcw
{\kappa}{\kappa}{2s+1}\le s+1\,,
\end{displaymath}
hence  $(\circ_s)$ implies both $(*_{2s})$ and $(*_{2s+1})$.

\medskip

\noindent{\bf First step:} $s=1$.\\
Take a $2$-dimensional vector space $V$ with $|V| = \kappa$ above
any field of cardinality ${\kappa}$ and let $\lcal$ be the family
of all lines (1-dimensional affine subspaces) in $V$. Then $\lcal$
is $2$-almost disjoint, hence a $(\kappa,\kappa,2)$-system, and it
trivially does not have a weak conflict free coloring with a
single color. So $\crcw {\kappa}{\kappa}2\ge 2=1+1$.

\medskip

\noindent{\bf Induction step:} $s\to (s+1) $.\\
Let $\fcal$ be a $(\kappa,\kappa,2s)$-system with $\wccf(\fcal)
\ge s+1$. We may then apply lemma \ref{lm:0*} to $\fcal$ with $t =
s+1$ to conclude that the $(\kappa,\kappa,2(s+1))$-system
$\fcal^*$ satisfies $w\ccf(\fcal^*) \ge t+1 = (s+1)+1$.
\end{proof}

Theorem \ref{tm:omegaw} shows that the upper bound established in
theorem \ref{tm:egeszresz} is sharp for $m = 0$. We shall show
next that this is also true for all $m > 0$, provided that GCH
holds and $d$ is odd. The following lemma plays the key role in
proving this.

\begin{lemma}\label{lm:l*}
For any cardinals $\lambda \ge \kappa \ge \omega$ and natural
number $\ell > 0$ we have a procedure assigning to any
$(\lambda,\kappa,2\ell+1)$-system $\fcal$ a
$\,(2^\lambda,\kappa,2\ell+1)$-system $\fcal^*$ so that
$\,w\ccf(\fcal) \ge \ell$ implies
$$w\ccf(\fcal^*) \ge w\ccf(\fcal) + \ell\,.$$
\end{lemma}

\begin{proof}
Fix the $(\lambda,\kappa,2\ell+1)$-system $\fcal$ and then choose
pairwise disjoint sets
\begin{displaymath}
 \{A_{\alpha}:{\alpha}<{\lambda}\}
\cup\{C_{\delta}:{\delta}<{\tla}\}\subs \br {\tla};{\lambda};\,.
\end{displaymath}
For ${\alpha}<{\lambda}$, resp. ${\delta}<{\tla}$,  let
$\acal_{\alpha}\subs \br A_{\alpha};{\kappa};$, resp.
$\ccal_{\delta}\subs \br C_{\delta};{\kappa};$, be isomorphic
copies of $\fcal$. For every $\delta < \tla$ we also fix a one-one
enumeration $\ccal_{\delta}=\{C_{{\delta},i}:i<{\lambda}\}$. Let
us then put
\begin{displaymath}
\scal=\{S\in  \big[\bigcup_{\alpha < \lambda}A_\alpha
\big]^{2\ell} :\forall {\alpha}<{\lambda}\ |S\cap A_{\alpha}|\le
1\}
\end{displaymath}
and $\{f_{\delta}:{\delta}<{\tla}\}$ be an enumeration of all
functions $f  : {\lambda} \to \scal$ that satisfy
\begin{displaymath}
\text{$f(i) \cap f(j) = \empt\,$ for any $\{i,j\} \in
[{\lambda}]^2$.}
\end{displaymath}
Finally, let $C^*_{{\delta},i}=C_{{\delta},i}\cup f_{\delta}(i)$
and put
\begin{displaymath}
 \fcal^*= \bigcup_{\alpha < \lambda} \acal_\alpha
\cup\{C^*_{{\delta},i} : {\delta}<{\tla},i<{\lambda}\}.
\end{displaymath}

\medskip

\noindent{ \bf Claim 1. }{\em $\fcal^*$ is $(2\ell+1)$-almost
disjoint. }

\smallskip
The only non-trivial case is showing $|C^*_{{\delta},i}\cap
C^*_{\delta',i'}|\le 2\ell$ for $\<{\delta},i\>\ne
\<{\delta}',i'\>$. Clearly, we have
\begin{displaymath}
C^*_{{\delta},i}\cap C^*_{{\delta}',i'}\subs (C_{{\delta}}\cap
C_{{\delta}'})\cup \bigl(f_{{\delta}}(i) \cap
f_{{\delta}'}(i')\bigr).
\end{displaymath}
Now, if ${\delta}\ne {\delta}'$ then $C_{{\delta}}\cap
C_{{\delta}'}=\empt$ and $|f_{{\delta}}(i) \cap
f_{{\delta}'}(i')|\le |f_{{\delta}}(i)|=2\ell$. If, on the other
hand, ${\delta}={\delta}'$ then $f_{{\delta}}(i) \cap
f_{{\delta}}(i')=\empt$ by definition, so $|C^*_{{\delta},i}\cap
C^*_{{\delta},i'}| = |C_{{\delta},i}\cap C_{{\delta},i'}|\le
2\ell\,$ because $\,\ccal_\delta\,$ is $(2\ell+1)$-almost
disjoint.

\medskip

\noindent{ \bf Claim 2. }If $\,w\ccf(\fcal) \ge \ell\,$ then
$w\ccf(\fcal^*) \ge w\ccf(\fcal) + \ell\,.$

\smallskip

The claim is obvious if $w\ccf(\fcal^*) \ge \omega$, so we may
assume that $w\ccf(\fcal^*) < \omega$. Now, let $h$ be any weak
conflict-free coloring of $\fcal^*$ with a finite color set $T$.
By $w\ccf(\acal_{\alpha})\ge \ell$, for each ${\alpha}<{\lambda}$
we have $|h[A_{\alpha}]|\ge \ell$, thus there are $I\in \br
{\lambda};{\lambda};$ and $M =\{m_j:j<\ell\}\in \br T;\ell;$ such
that $h[A_{\alpha}]\supset M$ for each ${\alpha}\in I$.

Let $\{{\alpha}_{{\zeta},n}:{\zeta}<{\lambda}\,,\,n<2\ell\}$ be
distinct elements of $I$. We may then, for each $j<\ell$, pick
$x_{{\zeta},2j}\in A_{{\alpha}_{{\zeta},2j}}$ and
$x_{{\zeta},2j+1}\in A_{{\alpha}_{{\zeta},2j+1}}$ satisfying
$$h(x_{{\zeta},2j})=h(x_{{\zeta},2j+1})=m_j\,.$$
There is ${\delta}<2^{\lambda}$ such that
$f_{\delta}({\zeta})=\{x_{{\zeta},n}:n<2\ell\}$ for all $\zeta <
\lambda$, then for each $m\in M$ and ${i}<{\lambda}$ we have
$|h^{-1}\{m\}\cap f_{\delta}({i})|\ge 2$. It follows that
$h\restriction (C_{\delta}\setm h^{-1}M)$ must be a weak conflict
free coloring of $\ccal_{\delta}$ with color set $T \setm M$,
showing that $|T \setm M| \ge w\ccf(\fcal)$, hence $|T | \ge
w\ccf(\fcal) + \ell$, completing the proof.
\end{proof}

\begin{theorem}\label{tm:step3}
For any $\kappa \ge \omega$ and $m,\ell \in \omega$ with $\ell >
0$ we have $$w\ccf(\beth_m(\kappa),\kappa,2\ell+1) \ge (m+1)\cdot
\ell + 1\,.$$
\end{theorem}

\begin{proof}
By Theorem \ref{tm:omegaw}, we have  $\crcw
{\kappa}{{\kappa}}{2\ell+1}=\ell +1.$ So we may simply apply
theorem \ref{tm:step3} $m$ times to obtain the result.
\end{proof}

As an immediate consequence of theorems \ref{tm:korlatos}
\ref{tm:step3} we obtain the following result.

\begin{theorem}\label{tm:step4}
For every infinite cardinal $\kappa$ and natural number $d > 1$ we
have $$\ccf(\beth_\omega(\kappa),\kappa,d) = \omega\,.$$
\end{theorem}

>From theorems \ref{tm:egeszresz} and \ref{tm:step3} we may
immediately deduce the promised exact value of $\crcw
{\kappa}{{\kappa}^{+m}}{2\ell+1}$ under GCH.

\begin{corollary}\label{cor:gchoup3}
If GCH holds then for any cardinal ${\kappa} \ge \omega$ and
integers $m\ge 0$,  $\ell>0$ we have
$$\crcw {\kappa}{{\kappa}^{+m}}{2\ell+1}=(m+1)\cdot\ell +1.$$
\end{corollary}

We do not know, in general, if an exact formula like this can be
obtained for $\crcw {\kappa}{{\kappa}^{+m}}{2\ell}$, but we do
know this in the simplest case $\ell =1\,.$ The key to this is
again a ``lift up" lemma in the spirit of lemmas \ref{lm:0*} and
\ref{lm:l*}.

\begin{lemma}\label{lm:2*}
For any $\lambda \ge \kappa \ge \omega$, we can assign to every
$(\lambda,\kappa,2)$-system $\fcal$ a
$\,(2^{2^\lambda},\kappa,2)$-system $\fcal^*$ so that if
$\,w\ccf(\fcal)$ is finite then
$$w\ccf(\fcal^*) > w\ccf(\fcal)\,.$$
\end{lemma}

\begin{proof}
Let $\fcal$ be any $(\lambda,\kappa,2)$-system and, to start with,
fix pairwise disjoint sets
\begin{multline}\notag
 \{A_{\delta}:{\delta}<{\tla}\}\cup
\{B_{{\eta},{\alpha}}:{\eta}<\ttla, {\alpha}<{\lambda}\}
\cup\\\{C_{{\eta},{\delta}}:{\eta}<\ttla,{\delta}<{\tla}\}\subs
\br \ttla;{\lambda};.
\end{multline}
For every ${\delta}<{\tla}$ let $\acal_{\delta}\subs \br
A_{\zeta};{\kappa};$ be an isomorphic copy of $\fcal$ and define
similarly $\bcal_{{\eta},{\alpha}}\subs \br
B_{{\eta},{\alpha}};{\kappa};$ and $\ccal_{{\eta},{\delta}} \subs
\br C_{{\eta},{\delta}};{\kappa};$. We also enumerate, without
repetitions, each $\ccal_{{\eta},{\delta}}$ as
$\{C_{{\eta},{\delta},i}:i<{\lambda}\}$.

Let us put $A=\cup \{A_{\delta}:{\delta}<{\tla}\}$ and
$B_{\eta}=\cup \{B_{{\eta},{\alpha}}: {\alpha}<{\lambda}\}$ for
each ${\eta}<\ttla$. Then enumerate the injective functions $f :
{\tla}\times {\lambda} \to A$ as $\{f_{\eta}:{\eta}<\ttla\}$ and,
for any ${\eta}<\ttla$, enumerate the  injective functions $g :
{\lambda} \to B_{\eta}$ as $\{g_{\eta,\delta}:{\delta}<{\tla}\}$.
Finally, let $$C^*_{{\eta},{\delta},i}=
C_{{\eta},{\delta},i}\cup\{f_{\eta}({\delta},i),g_{\eta,\delta}(i)\}$$
and put
\begin{multline}\notag
\fcal^*= \bigcup\{\acal_\delta : \delta < \tla \}\cup \bigcup
\{\bcal_{{\eta},{\alpha}} : \< \eta,\alpha \> \in \ttla \times \lambda \}\cup \\
\cup\{C^*_{{\eta},{\delta},i} :\<\eta,\delta,i\> \in \ttla \times
\tla \times {\lambda}\}.
\end{multline}

\medskip

\noindent{ \bf Claim 1. }{\em $\fcal^*$ is $2$-almost disjoint.}

\smallskip
The only non-trivial task is to show that
$|C^*_{{\eta},{\delta},i}\cap C^*_{{\eta}',{\delta}',i'}|\le 1$
for $\<{\eta},{\delta},i\>\ne \<{\eta}',{\delta}',i'\>$. Clearly,
we have
\begin{multline}
C^*_{{\eta},{\delta},i}\cap C^*_{{\eta}',{\delta}',i'}\subs
(C_{{\eta},{\delta},i}\cap C_{{\eta}',{\delta}',i'})\cup\\
\cup\bigl(\{f_{\eta}({\delta},i)\} \cap
\{f_{\eta'}({\delta}',i')\}) \cup \bigl(\{g_{\eta,\delta}(i)\}
\cap \{g_{\eta',\delta'}(i')\}) \bigr).
\end{multline}
If ${\eta}\ne {\eta}'$ then $C_{{\eta},{\delta}}\cap
C_{{\eta}',{\delta}'}=\empt$, and $\{g_{\eta,\delta}(i)\} \cap
\{g_{\eta',\delta'}(i')\}\subs B_{{\eta}}\cap B_{{\eta}'}=\empt$,
hence $$C^*_{{\eta},{\delta},i}\cap
C^*_{{\eta}',{\delta}',i'}\subs \{f_{\eta}({\delta},i)\} \cap
\{f_{\eta'}({\delta}',i')\}.$$ If ${\eta}={\eta}'$ and
${\delta}\ne {\delta}'$ then $C_{{\eta},{\delta}}\cap
C_{{\eta},{\delta}'}=\empt$, and $f_{\eta}({\delta},i)\ne
f_{\eta}({\delta}',i')$ because $f_{\eta}$ is injective, hence
$$C^*_{{\eta},{\delta},i}\cap C^*_{{\eta}',{\delta}',i'}\subs
\{g_{\eta,\delta}(i)\} \cap \{g_{\eta',\delta'}(i')\}.$$ Finally,
if ${\eta}={\eta}'$, ${\delta}={\delta}'$, and $i\ne i'$ then
$f_{\eta}({\delta},i)\ne f_{\eta}({\delta},i')$ and
$g_{\eta,\delta}(i)\ne g_{\eta,\delta}(i')$ because
$g_{\eta,\delta}$ is also injective, and so
$$|C^*_{{\eta},{\delta},i}\cap C^*_{{\eta}',{\delta}',i'}| =
|C_{{\eta},{\delta},i}\cap C_{{\eta}',{\delta}',i'}|\le 1\,.$$

\medskip

\noindent{ \bf Claim 2. }$w\ccf(\fcal^*) > w\ccf(\fcal)\,$ {\em if
the latter is finite.}

\smallskip
Assume that $w\ccf(\fcal) = k < \omega$ and, contrary to our
claim, $h$ is a weak conflict-free coloring of $\fcal^*$ with
$\ran(h)=k$. Then, for each ${\delta}<{\tla}$, the equality
$\,w\ccf(\acal_{\delta}) = k$ implies that there is $a_{\delta}\in
A_{\delta}$ with $h(a_{\delta})=0$. Since
$|\{a_{\delta}:{\delta}<{\tla}\}| = {\tla}$, there is an
${\eta}<\ttla$ with $\ran (f_{\eta})=
\{a_{\delta}:{\delta}<{\tla}\}\subs h^{-1}\{0\}$.

Fix this $\eta$ and then apply $\,w\ccf(\bcal_{{\eta},{\alpha}}) =
k$ to find, for each ${\alpha}<{\lambda}$, some $b_{\alpha}\in
B_{{\eta},{\alpha}}$ with $h(b_{\alpha})=0$. Again, we have
$|\{b_{\alpha}:{\alpha}<{\lambda}\}| = {\lambda}$, hence there is
a $\,{\delta}<{\tla}$ with $\ran(g_{\eta,\delta})=
\{b_{\alpha}:{\alpha}<{\lambda}\}\subs h^{-1}\{0\}$.

But then for each $i<{\lambda}$ we have $\{f_{\eta}({\delta},i),
g_{\eta,\delta}(i)\} \subs h^{-1}\{0\}$, consequently
$h\restriction (C_{{\eta},{\delta}}\setm h^{-1}\{0\})$ must be a
weak conflict free coloring of $\ccal_{\eta,\delta}$ with $k-1$
colors, a contradiction. This contradiction proves Claim 2 and
completes the proof of the lemma.
\end{proof}

\begin{theorem}\label{tm:step3'}
For any $\kappa \ge \omega$ and $\,m \in \omega$ we have
$$w\ccf(\beth_m(\kappa),\kappa,2) \ge \left\lfloor \frac {m}{2}\right\rfloor+2\,.$$
\end{theorem}

\begin{proof}
By theorem \ref{tm:omegaw} this is true for $m = 0$ and $m = 1$.
Moreover, if we assume $\,w\ccf(\beth_m(\kappa),\kappa,2) \ge
\left\lfloor \frac {m}{2}\right\rfloor+2\,$ then, applying lemma
\ref{lm:2*} with $\lambda = \beth_m(\kappa)$, we obtain
$$w\ccf(\beth_{m+2}(\kappa),\kappa,2) \ge \left\lfloor \frac
{m}{2}\right\rfloor+2+1 = \left\lfloor \frac
{m+2}{2}\right\rfloor+2\,.$$ Thus the theorem follows by a
straight-forward induction.
\end{proof}

Comparing this with theorem \ref{tm:egeszresz} we get the
following result.

\begin{corollary}\label{cor:gchoup2}
For ${\kappa} \ge \omega$ and $m \in \omega$, the equality
$\beth_m(\kappa) = \kappa^{+m}$ implies
\begin{equation}\label{eq:e1}
\crcw {{\kappa}}{{\kappa}^{+m}}2=\left\lfloor \frac
{m}{2}\right\rfloor+2\,.
\end{equation}

\end{corollary}

\medskip

\section{Attempts to compute $\ccf(\omega_k,\omega,2)$}

In the previous section we succeeded in computing the exact value
of $\,\crcw {{\kappa}}{{\kappa}^{+m}}d$ in a lot of cases, at
least under GCH. As we have $$w\ccf(\lambda,\kappa,d) \le
\ccf(\lambda,\kappa,d) \le w\ccf(\lambda,\kappa,d)+1\,,$$ this
gives us a lot of information about $\crc
{{\kappa}}{{\kappa}^{+m}}d$ as well. But can we find the exact
value of $\,\crc {{\kappa}}{{\kappa}^{+m}}d$, or even just of
$\,\crc {{\omega}}{\omega_m}d$, say under GCH and for many values
of $m$ and $d$? This turned out to be a very hard problem that we
address in the present section, admittedly with only rather meager
results. There is no problem in the simplest possible case: $m \le
1$ and $d = 2$.

\begin{proposition}\label{f:omega}
$\crc{\kappa}{\kappa}2=\crc{\kappa}{\kappa^+}2=3$ for all
${\kappa}\ge {\omega}$.
\end{proposition}

\begin{proof}
First, by theorem \ref{tm:egeszresz}, we have
$$\crc{\kappa}{\kappa}2 \le \crc{\kappa}{\kappa^+}2 \le 3\,.$$ We have seen in
the proof of theorem \ref{tm:omegaw} that if $V$ is any
$2$-dimensional vector space with $|V| = \kappa$ above any field
of cardinality ${\kappa}$, then the $(\kappa,\kappa,2)$-system
$\lcal$ of all lines in $V$ satisfies
$$w\ccf(\lcal) = w\ccf(\kappa,\kappa,2) = 2\,.$$ Consequently, we shall be
done if we can show that $\lcal$ does not have a conflict free
coloring with 2 colors.

Assume, on the contrary, that $f:V\to 2$ is a CF-coloring of
$\lcal$ and write $C_i=f^{-1}\{i\}$ for $i\in 2$. Since $|C_i\cap
L|\ge 1$ for each line $L$ and color $i<2$, neither $C_i$ is
collinear, i.e. $C_i\not\subset L$ for any $i<2$ and for any line
$L$. Thus there are four lines $\{K^j_i:i,j<2\} \subs \lcal$ such
that $|C_i\cap K^j_i|\ge 2$ for all $i,j < 2$. Since $f$ is a
CF-coloring, for any $i,j<2$ we have a point $P^j_i$ with $K^j_i
\cap C_{1-i} = \{P^j_i\}$.

There is a line $L$ that intersects each $K^j_i$ in distinct
points which are all different from the points $P^j_i$. Then
$|L\cap C_i|\ge 2$ for $i < 2$, hence $f$ is not a CF-coloring of
$\lcal$, a contradiction.
\end{proof}

What can we say about $\,\crc {\omega}{\omega_m}2$ for $m>1$? If
$\beth_m = \omega_m$, in particular under GCH, from corollary
\ref{cor:gchoup2}, we have, for any $m < \omega$,
$$\left\lfloor \frac {m}{2}\right\rfloor+2 \le \crc {{\omega}}{{\omega}_{m}}2
\le \left\lfloor \frac {m}{2}\right\rfloor+3\,,$$  hence, in
particular,
\begin{displaymath}
 3\le\crc {\omega}\oot2\le \crc {\omega}\ooh2  \le 4.
\end{displaymath}

We actually do not know the exact value of $\crc {\omega}\oot2$
even under GCH, but we can reformulate the problem in terms of the
strict five-parameter arrow relation that was introduced in
definition \ref{def:gen_princ}. One direction of this works in
ZFC.

\begin{theorem}\label{tm:equi}
If $[\kappa,\kappa,2,2] \to 3$  then $\ccf(\kappa^{++},\kappa,2) =
3$.

\end{theorem}

\begin{proof}
Starting with the relation $[\kappa,\kappa,2,2] \to 3$ and
applying theorem \ref{tm:ind_new} (2) twice we obtain
$[\kappa^{++},\kappa,2,0] \to 3$ which, of course, is just
$[\kappa^{++},\kappa,2] \to 3$, and hence, together with
$\ccf(\kappa,\kappa,2) = 3$, implies $\ccf(\kappa^{++},\kappa,2) =
3$.
\end{proof}

To go in the opposite direction, we first need the following
result concerning the relation $[\lambda,\kappa,2,k] \to x$.

\begin{lemma}\label{lm:x}
If $\,\ninduu {\kappa}{\lambda}2k{x}$ then this can be witnessed
by a $(\lambda,\kappa,2)$-system $\xcal = \{X_i : i < \lambda\}
\subs [\lambda]^\kappa$ and a map $c \in \fcal(\lambda,x)$ such
that $$Y = \dom(c) = \cup\{Y_i : i < \lambda\},$$ where $X_i \cap
Y \subs Y_i \in [Y]^k$ for each $i < \lambda$ and the $k$-element
sets $Y_i$ are pairwise disjoint.
\end{lemma}

\begin{proof}
Fix an arbitrary $(\lambda,\kappa,2)$-system $\xcal = \{X_i : i <
\lambda\} \subs [\lambda]^\kappa$ and a map $c \in
\fcal(\lambda,x)$ that witnesses $\,\ninduu
{\kappa}{\lambda}2k{x}$. For each $y \in Y$ consider the set $I_y
= \{ i \in \lambda : X_i \cap Y \ne \emptyset\}$ and if $|I_y| >
1$ then, for each $i \in I_y$ replace $y$ in $X_i$ by the pair
$\<y,i\>$ and ``blow up" $y$ in $Y$ to $I_y \times \{y\}$. Having
done this for all $y \in Y$ let us denote the ``new" $X_i$ by
$X'_i$ and the ``new" $Y$ by $Y'$. Also define the ``new" function
$c'$ on $Y'$ by the rule $c'(\<y,i\>) = c(y)$. We may then add, if
necessary, completely new elements to $Y'$ (and extend $c'$ to
them arbitrarily) to obtain the pairwise disjoint $k$-element sets
$Y_i \supset X'_i \cap Y'$ forming a partition of $Y'$.

It is easy to check that the $(\lambda,\kappa,2)$-system $\xcal' =
\{X'_i : i < \lambda \}$ and the map $c'$, that now are of the
desired form, also witness $\ninduu {\kappa}{\lambda}2k{x}$.

\end{proof}

\begin{theorem}\label{tm:x}
For any $\lambda \ge \kappa \ge \omega > k\,$,
$$\,\ccf(\lambda,\kappa,2) = \ccf(\beth_k(\lambda),\kappa,2) = x <
\omega$$ implies $\,[\lambda,\kappa,2,k] \to x$.
\end{theorem}

\begin{proof}
By the previous result, to conclude $[\lambda,\kappa,2,k] \to x$,
it suffices to show the existence of a conflict free coloring of
$\xcal$ that extends $c$ for any $(\lambda,\kappa,2)$-system
$\xcal=\{X_i:i<{\lambda}\}\subs \br \lambda;{\kappa};$ and partial
map $c \in \fcal(\lambda,x)$ satisfying the conditions of lemma
\ref{lm:x}. That is, we may assume having a partition $\{Y_i : i <
\lambda\}$ of $\dom(c) = Y$ into disjoint $k$-element sets such
that $X_i \cap Y \subs Y_i$ for all $i < \lambda$. For each $i <
\lambda$ we write $Y_i = \{y_i^j : 1 \le j \le k\}$. By
$\,\ccf(\lambda,\kappa,2) = x$, we can fix a
$(\lambda,\kappa,2)$-system $\fcal$ with $\ccf(\fcal) = x$.

We now introduce some notation. For any $j$ we write
$\beth_j(\lambda) = \lambda_j$ (so, in particular, $\lambda_0 =
\lambda$) and put $\Pi = \lambda_k \times
\lambda_{k-1}\times...\times \lambda_0$. For each $j \le k$ we
shall also write $\Pi^j = \lambda_k \times...\times
\lambda_{j+1}\times \lambda_{j-1} \times...\times \lambda_0$, that
is the members of $\Pi^j$ are obtained from the members of $\Pi$
by deleting their $j$-coordinate.

Next we choose pairwise disjoint sets $\{A_\sigma^j : j \le
k,\,\sigma \in \Pi^j\}$ of size $\lambda$, and for every $j$ with
$1 \le j \le k$ and $\sigma \in \Pi^j$ we let $\acal_\sigma^j$ be
a copy of $\fcal$ on $A_\sigma^j$.

For fixed $j$ with $1 \le j \le k$ and $\varrho \in \lambda_k
\times... \times\lambda_{j+1}$, consider the family
$\mathbb{F}_\varrho^j$ of all functions $f$ such that $\dom(f) =
\lambda_{j-1} \times...\times \lambda_0\,$ and $f(\eta) \in
A_{\varrho\smallfrown\eta}^j \,$ for all $\,\eta \in \lambda_{j-1}
\times...\times \lambda_0\,$. Then $|\mathbb{F}_\varrho^j| =
\lambda_j$, hence for every $j$ with $1 \le j \le k$ there is a
function $f^j$ with $\dom(f^j)= \Pi$ and having the property that,
if we fix $\varrho \in \lambda_k \times... \times\lambda_{j+1}$,
then the functions $\eta \mapsto f^j(\varrho\smallfrown
\<\xi\>\smallfrown \eta)$ enumerate $\mathbb{F}_\varrho^j$ in a
one-one manner, as $\xi$ ranges over $\lambda_j$.

For any $\sigma \in \Pi^0$ we put $$B_\sigma^0 = A_\sigma^0 \cup
\{f^j(\sigma\smallfrown \<i\>) : 1 \le j \le k \mbox{ and } i <
\lambda
 \}\,.$$ Then, as $|\lambda \setm Y| = \lambda$, we may fix a bijection $h_\sigma : \lambda \to B_\sigma^0$
such that $$h_\sigma[\lambda \setm Y] = A_\sigma^0 \mbox{ and }
h_\sigma(y_i^j) = f^j(\sigma\smallfrown \<i\>)$$ for any $1 \le j
\le k$ and $i < \lambda$. Now, if $\tau \in \Pi$ with $\tau =
\sigma\smallfrown <i>$ then we set $B_\tau = h_\sigma[X_i]$.

We claim that the family $$\acal = \bigcup \{ \acal_\sigma^j : 1
\le j \le k \mbox{ and } \sigma \in \Pi^j \} \cup \{B_\tau : \tau
\in \Pi\}$$ is 2-almost disjoint. Here the only problematic task
is to show that $|B_\tau \cap B_{\tau'}| \le 1$ for two distinct
members, $\tau = \<\xi_k, ...,\xi_1,i \>$ and $\tau'= \<\xi'_k,
...,\xi'_1,i' \>$, of $\Pi$. Let $\sigma = \<\xi_k, ...,\xi_1 \>$
and $\sigma' =\<\xi'_k, ...,\xi'_1 \>\,$. If $\sigma \ne \sigma'$
and $j\ge 1$ is maximal such that $\xi_j \ne \xi'_j$, then we have
$B_\tau \cap B_{\tau'} \subs \{f^j(\tau)\} \cap \{f^j(\tau')\}$.
If, however, $\sigma = \sigma'$ then $i \ne i'$ and $$B_\tau \cap
B_{\tau'} = h_\sigma[X_i] \cap h_\sigma[X_{i'}] = h_\sigma[X_i
\cap X_{i'}]\,,$$ hence we are done because $\xcal$ is 2-almost
disjoint.

Thus $\acal$ is a $(\lambda_k\,,\kappa,2)$-system and so, by our
assumption, it has a conflict free coloring $d : \cup \acal \to
x$. Our choice of $\acal_\varrho^j$ implies that, for every $j$
with $1 \le j \le k$ and $\varrho \in \Pi^j$, we have
$d[A_\varrho^j] = x$. It follows that there is a function $f \in
\mathbb{F}^k_\emptyset$ which satisfies $d(f(\varrho)) = c(y_i^k)$
for all $\varrho \in \Pi^k$, where $i$ is the last ($0$)
coordinate of $\varrho$, and there is an ordinal $\xi_k <
\lambda_k$ for which we have $f(\varrho) =
f^k(\<\xi_k\>\smallfrown \varrho)$ for all $\varrho \in \Pi^k$.

Repeating this procedure ``downward", step by step, we arrive at a
sequence $\sigma = \<\xi_k, ... ,\xi_1 \> \in \Pi^0 $ which, for
any $j$ with $1 \le j \le k$ and $i < \lambda$, satisfies the
equality $$d(f^j(\sigma\smallfrown \<i\>)) = c(y_i^j)\,.$$ But
recall that we have $h_\sigma(y_i^j) = f^j(\sigma\smallfrown
\<i\>)$ by definition, hence the composition $d\circ h_\sigma$ is
a conflict free coloring of $\xcal$ which extends $c$, completing
our proof of $\,[\lambda,\kappa,2,k] \to x$.
\end{proof}

\begin{corollary}\label{cor:equi2}
For every infinite cardinal $\kappa\,$,
$\ccf(\beth_2(\kappa),\kappa,2) = 3$ implies $[\kappa,\kappa,2,2]
\to 3$. Consequently, if $\,\beth_2(\kappa) = \kappa^{++}$, in
particular under GCH, $[\kappa,\kappa,2,2] \to 3$ is equivalent to
$\ccf(\kappa^{++},\kappa,2) = 3$.
\end{corollary}

Our next aim is to show that $\crc {\omega}\ooh2  = 4$ under GCH.
This will follow from the ZFC result $\crc {\omega}\oohp2 \ge 4$
that, in turn, follows from the negative relation $\nindu
{\omega}2033$. To prove the latter, we need the following
technical lemma.

\begin{lemma}\label{lm:fini}
There are a finite $2$-almost disjoint family $\acal$ of countably
infinite sets, a finite set $C$, and a function $c:C\to 3$ such
that
\begin{enumerate}[(1)]
\item  $|A \cap C|= 4$ for each $A\in\acal$,
\item the sets $\{A\cap C:A\in\acal\}$ are pairwise disjoint,
\item
$c$ can not be extended to a
 conflict free coloring of $\acal$ with $3$ colors.
\end{enumerate}
 \end{lemma}

\begin{proof}
For $\{a, b\} \in [\mathbb R]^2$ let $L_{a,b}$ be the line in
$\mathbb R^2$ which contains $a$ and $b$ and put
$E_{a,b}=L_{a,b}\cap \mathbb Z^2$. We then put
\begin{displaymath}
\acal=\{E_{a,b}:  \{a,b\}\in \br 4\times 6;2;\}.
\end{displaymath}
Let $C\subs \cup \acal \setm (4\times 6)$ be any finite set that
satisfies (1) and (2).

Write $V_i=E_{\<i,0\>,\<i,1\>}$ for $i<4$ and
$H_j=E_{\<0,j\>,\<1,j\>}$ for $j<6$. Define $c:C\to 3$ in such a
way that if $C_i=c^{-1}\{i\}$ for $i<3$, then we have
\begin{enumerate}[(a)]
\item for each $i<4$
   \begin{displaymath}
|C_0\cap  V_i|= |C_1\cap  V_i|=2
   \end{displaymath}
\item for each $j<6$
   \begin{displaymath}
|C_1\cap  H_j|= |C_2\cap  H_j|=2
  \end{displaymath}
 \item for each $i\ne i'<4$ and $j\ne j' <6$
   \begin{displaymath}
|C_0\cap  E_{\<i,j\>, \<i',j'\>}|= |C_2\cap  E_{\<i,j\>,
\<i',j'\>}|=2
   \end{displaymath}
\end{enumerate}

Assume that $f : \cup \acal \to 3$ is a conflict free coloring of
$\acal$ with $c\subset f$. Then, by (a), for each $i<4$ there is
exactly one $x_i\in V_i$ such that $f(x_i)=2$. Since $6-4=2$
 there are $j\ne j'<6$ such that
$$\{x_i:i<4\}\cap (H_j\cup H_{j'})=\empt\,.$$
By (b), there are unique $y_j\in H_j$ and $y_{j'}\in H_{j'}$,
respectively, such that $f(y_j)=f(y_{j'})=0$. Since $4-2=2$
 there are $i\ne i'<4$ such that
$$\{y_j, y_{j'}\}\cap (V_i\cup V_{i'})=\empt\,.$$

Let $a=\<i,j\>$ and $b=\<i'j'\>$. Then $a\ne x_i$ implies $f(a)\ne
2$ and  similarly,  $a\ne y_j$ implies $f(a)\ne 0$, hence
$f(a)=1$. Similarly, we have $f(b)=1$. But, as $a,b\in E_{a,b}$
and (c) holds, we have $|E_{a,b}\cap f^{-1}\{i\}| > 1$ for each $i
< 3$, which is a contradiction.
\end{proof}

\begin{theorem}\label{tm:ooh}
$\nindu {\omega}2033$.
\end{theorem}

\begin{proof}
We shall  construct a $2$-almost disjoint family $\hcal \subs
[H]^\omega$ for a countable set $H$, a subset $K\subs H$, and a
function $d:K\to 3$ such that
\begin{enumerate}[(1)]
 \item  $|H\cap K|\le 3$ for each $H\in\hcal$,
\item
$d$ can not be extended to a
 conflict free coloring  of $\hcal$ with $3$ colors.
\end{enumerate}

We first choose, using $\crc {\omega}{\omega}2=3$,  a $2$-almost
disjoint family $\bcal\subs \br {\omega};{\omega};$ such that
\begin{multline}\label{eq:sok_szin}
\text{if $f:{\omega}\to 3$ is any conflict-free coloring of
$\bcal$}\\\text{then $f^{-1}\{i\}$ is infinite for each $i<3$.}
\end{multline}
(Let $\{A_n : n < \omega\}$ be a partition of $\omega$ into
infinite sets and $\bcal_n \subs [A_n]^\omega$ be a copy of a
family witnessing $\crc {\omega}{\omega}2=3$. Then $\bcal =
\cup_{n<\omega} \bcal_n$ clearly satisfies (\ref{eq:sok_szin}).)

Fix a countable set $X$,  a finite family $\acal\subs \br
X;{\omega};$, a finite set $C\subs X$, and a function $c:C\to 3$
as in Lemma \ref{lm:fini} . $D\subs C$ be such that $|A\cap D|=1$
for each $A\in \acal$.

Let $\gcal$ denote the collection of all injective functions
$g:D\stackrel{1-1}{\longrightarrow}{\omega}$ and $\{H_g:g\in
\gcal\}$ be disjoint countably infinite sets with $H_g\cap
{\omega}=\empt$. For each $g\in \gcal\,$ fix a bijection
$h'_g:(X\setm D)\to H_g$ and put $h_g=g\cup h'_g$.

Let us then define
\begin{gather}
H=\omega\cup \bigcup \{H_g:g\in \gcal\},\\
\hcal=\bcal\cup\{h_g[A]:A\in \acal,g\in \gcal\},\\
K=\cup\{h_g[C\setm D]:g\in \gcal\},
\end{gather}
and, finally, define $d:K\to 3$ as follows:
\begin{equation}\label{eq:K}
\text{if $k=h_g(x)$ for some $x\in C\setm D$ and $g \in \gcal$,
then $d(k)=c(x)$.}
\end{equation}
We claim that $H$, $\hcal$, $K$, and $d$ are as required, that is
satisfy (1) and (2). Of course, only (2) needs to be checked.

Assume, on the contrary, that $f:H\to 3$ is a conflict-free
coloring for $\hcal$ with $d\subs f$. Using (\ref{eq:sok_szin}) we
may find an injective function $g:D\to {\omega}$ such that for
each $x \in D$ we have
\begin{equation}\label{eq:fgc}
f(g(x))=c(x).
\end{equation}
Let us now define $F:{\omega}\to 3$ by $F(x)=f(h_g(x))$. Since $f$
is a conflict free coloring of $\{h_g[A]:A\in \acal\}\subs \hcal$
and $h_g$ is a bijection, $F$ is a  conflict free coloring of
$\acal$.

If $x\in D$ then $F(x)=f(h_g(x))=f(g(x))=c(x)$ by (\ref{eq:fgc})
and if $x\in C\setm D$ then $F(x)=f(h_g(x))=d(h_g(x))=c(x)$ by
(\ref{eq:K}), hence $c\subs F$. But this contradicts the choice of
$\acal$, which proves that $H$, $K$, $\hcal$, and $d$ really
satisfy conditions (1) and (2).
\end{proof}
\smallskip
\begin{corollary}\label{cor:beth3}
$\ccf(\beth_3,\omega,2) \ge 4$. Consequently, if $\beth_3 =
\omega_3$ then $\ccf(\omega_3,\omega,2) = 4$.
\end{corollary}

\smallskip

\begin{problem}
Is $\,\ccf(\omega_2,\omega,2) = 4\,$ provable under GCH?
\end{problem}

\bigskip

\bigskip
\begin{center}
\sc Part III. The case $\lambda \ge \kappa \ge \omega = \mu$
\end{center}

\section{Consistent upper bounds for $\ccf(\lambda,\kappa,\omega)$}

We start by pointing out that $\ccf(\lambda,\kappa,\omega)$ is
always infinite. This follows immediately from the next
proposition because $\ccf(\lambda,\kappa,\omega)$ is increasing in
its first parameter.

\begin{proposition}\label{pr:kko}
For every infinite cardinal $\kappa$ we have
$$\ccf(\kappa,\kappa,\omega) \ge \omega.$$
\end{proposition}

\begin{proof}
By theorem \ref{tm:omegaw}, for every $d \in \omega \setm 2$ there
is a $(\kappa,\kappa,d)$-system $\acal_d$ such that
$$\ccf(\acal_d) \ge w\ccf(\acal_d) = {\left\lfloor {\frac{d}2}   \right\rfloor+1}.$$
But clearly if $\acal$ is the union of $\{\acal_d : d \in \omega
\setm 2 \}$ (taken on disjoint underlying sets) then $\acal$ is a
$(\kappa,\kappa,\omega)$-system with $\ccf(\acal) \ge \omega$.
\end{proof}

The main aim of this section is to show that we have
$\ccf(\lambda,\kappa,\omega) \le \omega_2$ for $\lambda \ge \kappa
\ge \omega_2$, provided that $\mu^\omega = \mu^+$ holds for every
$\mu < \lambda$ with $\cf(\mu) = \omega$. Moreover, if in addition
$\square_\mu$ also holds for any $\mu$ with $\cf(\mu) = \omega <
\mu < \lambda$, then we even have $\ccf(\lambda,\kappa,\omega) \le
\omega_1$ whenever $\lambda \ge \kappa \ge \omega_1$. The first
part will follow from a general stepping up result, whose
formulation needs the following definition.

\begin{definition}
Assume that $\omega \le {\rho} \le {\lambda}$ are  cardinals,
$\acal$ is any set-system, and $\vec N =
\<N_{\alpha}:{\alpha}<{\lambda}\>$ is a ${\lambda}$-chain of
elementary submodels. We say that  $\vec N$ {\em ${\rho}$-cuts}
$\acal$ iff
\begin{equation}\label{eq:dec}
\text{$\acal\in N_1$, moreover ${\alpha}<{\lambda}$ and $A\in
\acal\setm N_{\alpha}$ imply $|A\cap N_{\alpha}|<{\rho}$.}
\end{equation}
\end{definition}

\begin{theorem}\label{tm:gen_step_up}
Let  ${\omega}\le {\mu}\le {\rho}\le {\kappa}\le {\lambda}$ be
cardinals and assume that every $(\lambda,\kappa,\mu)$-system is
$\rho$-cut by a ${\lambda}$-chain of elementary submodels. Assume
also that
\begin{enumerate}[(i)]
\item if ${\kappa}={\lambda}$ then there is ${\kappa}^*<{\kappa}$ such that
$[\kappa',\kappa',\mu] \Rightarrow \rho$ whenever
${\kappa}^*\le{\kappa'}<{\kappa}$\,  (note that in this case $\rho
\le \kappa^* < \kappa = \lambda$),
\item if ${\kappa}<{\lambda}$ then $[\lambda',\kappa',\mu] \Rightarrow \rho$ whenever ${\kappa}\le
\kappa' \le {\lambda}'<{\lambda}$.
\end{enumerate}
Then $[\lambda,\kappa,\mu] \Rightarrow \rho$.
\end{theorem}

\begin{proof}
Let $\acal \subs \br {\lambda};{\kappa};$ be a
$(\lambda,\kappa,\mu)$-system and let $\vec
N=\<N_{\alpha}:{\alpha}<{\lambda}\>$ be a ${\lambda}$-chain of
elementary submodels which ${\rho}$-cuts $\acal$. We can assume
that $\max({\kappa}^*+1,{\rho}+1)\subs N_1$ in case
${\kappa}={\lambda}$ and ${\kappa}+1\subs N_1$ in case
${\kappa}<{\lambda}$. For each ${\alpha}<{\lambda}$ let
\begin{displaymath}
  \acal_{\alpha}=\acal\cap (N_{{\alpha}+1}\setm N_{\alpha})\,,
\end{displaymath}
then $\<\acal_{\alpha}:{\alpha}<{\lambda}\>$ is a partition of
$\acal$ and $|\acal_{\alpha}|\le |N_{{\alpha}+1}|<{\lambda}$. We
let
\begin{equation}\label{eq:ya}
Y_{\alpha}={\lambda}\cap N_{{\alpha}+1}\setm \bigl(N_{\alpha} \cup
 \bigcup\acal\cap N_{\alpha}\bigr)
\end{equation}
and
\begin{displaymath}
\acal'_{\alpha}=\{A\cap Y_{\alpha}:A\in \acal_{\alpha}\}.
\end{displaymath}
If $A\in \acal_{\alpha}$ then $|A\cap N_{\alpha}|<{\rho} \le
\kappa$, hence
\begin{equation}\label{eq:ayr}
| A\cap \cup\{Y_{\beta}:{\beta}<{\alpha}\}|<{\rho},
\end{equation}
and, by definition,
\begin{equation}\label{eq:aydis}
A \cap\cup\{Y_{\beta}:{\beta}>{\alpha}\}=\empt\,.
\end{equation}

Assume first that ${\kappa}={\lambda}$. Then $A \in \acal_\alpha$
implies $$|A\cap \bigcup (\acal \cap N_{\alpha})|\le
{\mu}\cdot|N_{\alpha}|<{\kappa}\,,$$ hence, by elementarity, $|A
\cap Y_\alpha|=|Y_{\alpha}|=|N_{{\alpha}+1}| \ge \kappa^*$.
Consequently, $\acal'_{\alpha}\subs \br Y_{\alpha};|Y_{\alpha}|;$
is a $(|Y_{\alpha}|,|Y_{\alpha}|,\mu)$-system and thus, by (i),
there is a function $c_{\alpha}:Y_{\alpha}\to {\rho}$ such that
for each $A\in \acal_{\alpha}$ we have
\begin{equation}\label{eq:ca1}
|{\rho}\setm I_{c_\alpha}(A \cap Y_\alpha)|<{\rho}.
\end{equation}

Assume now that  ${\kappa}<{\lambda}$. Then $\cup(\acal\cap
N_{\alpha})\subs N_{\alpha}$, and so $$A\cap Y_{\alpha} = A \setm
A \cap N_\alpha \in [Y_\alpha]^\kappa\,.$$ But ${\kappa}\le
|Y_{\alpha}|=|N_{{\alpha}+1}|<{\lambda}$ and $\acal'_{\alpha}\subs
\br Y_{\alpha};{\kappa};$ is ${\mu}$-almost disjoint, so by (ii)
there is $c_{\alpha}:Y_{\alpha}\to {\rho}$ such that for each
$A\in \acal_{\alpha}$ we have
\begin{equation}\label{eq:ca2}
|{\rho}\setm I_{c_\alpha}(A \cap Y_\alpha)|<{\rho}.
\end{equation}

Let us put (in both cases)
$c=\cup\{c_{\alpha}:{\alpha}<{\lambda}\}$, then $c \in
\fcal({\lambda},{\rho})$. For $A\in \acal$ pick
${\alpha}<{\lambda}$ with $A\in \acal_{\alpha}$, then
(\ref{eq:aydis}) implies
$$I_c(A)\supset I_{c_{\alpha}}(A\cap Y_{\alpha})\setm c[A\cap
\cup\{Y_{\beta}: {\beta}<{\alpha}\}]\,. $$ But $|A\cap
\cup\{Y_{\beta}: {\beta}<{\alpha}\} |<{\rho}$ by (\ref{eq:ayr}),
hence either (\ref{eq:ca1}) or (\ref{eq:ca2}) implies $|\rho \setm
I_c(A)| < {\rho}$. Finally, if $\dom(c) \ne \lambda$ then we may
extend $c$ to a full function $d : \lambda \to \rho+1$ by mapping
every member of $\lambda \setm \dom(c)$ to $\rho$, and then we
have $|\varrho \setm I_{d}(A)| < \rho$, which completes the proof
of $[\lambda,\kappa,\mu] \Rightarrow \rho$.
\end{proof}

Now, using the trivial relation $[\rho,\rho,\mu] \Rightarrow \rho$
and theorem \ref{tm:gen_step_up}, the following result may be
established by a straight-forward transfinite induction. The
details are left to the reader.

\begin{corollary}\label{cor:gen_step_up}
Let ${\omega}\le {\mu}\le {\rho} < {\lambda}$ be cardinals. If
every $(\lambda',\kappa,\mu)$-system is $\rho$-cut by a
${\lambda}'$-chain of elementary submodels whenever $\rho <
\lambda' \le \lambda$ and $\rho \le \kappa \le \lambda'$ then
$\,\Carr {\kappa}{\lambda}{\mu}{\rho}$.
\end{corollary}

The following easy lemma will be used in the proof of the first
result that was promised in the introductory paragraph of this
section.

\begin{lemma}\label{lm:gen_small_close}
Assume that $\lambda \ge \omega_2$ and ${\mu}^{\omega}={\mu}^+$
holds for each ${\mu}<{\lambda}$ with $\cf({\mu})={\omega}$. If
$\acal$ is an $\omega$-almost disjoint set system and $\,X$ is any
set with $\, |X| < \lambda$, then
\begin{displaymath}
 \big|\{A\in \acal:|X\cap A| > \omega \}\big| \le |X|.
\end{displaymath}
\end{lemma}

\begin{proof}
It obviously follows from our assumption that if $\mu < \lambda$
and $\cf(\mu) > \omega$ then $\mu^\omega = \mu$. Thus, if
$\cf(|X|) > \omega$ then, as $\acal$ is $\omega$-almost disjoint,
we even have $$|\{A\in \acal:|X\cap A|\ge \omega\}|\le
|X|^{\omega}= |X|\,.$$ If, however, $\cf(|X|)={\omega} < |X|$ then
we may write $X=\cup\{X_n:n<{\omega}\}$ with $|X_n|<|X|$ for each
$n < \omega$. But then we have
\begin{displaymath}
 |\{A\in \acal:|X\cap A|\ge \oo\}|=
|\{A\in \acal:\exists n\, |X_n\cap A|\ge \oo\}|,
\end{displaymath}
and so
\begin{multline}\notag
|\{A\in \acal: |X\cap A|\ge \oo\}|\le \sum_{n<{\omega}}|\{A\in
\acal:
|X_n\cap A|\ge {\omega}\}|\le\\
\le \sum_{n<{\omega}}|X_n|^{\omega} = |X|.
\end{multline}
\end{proof}

\begin{theorem}\label{tm:above_oot}
Assume that $\lambda \ge \omega_2$ and ${\mu}^{\omega}={\mu}^+$
holds for each ${\mu}<{\lambda}$ with $\cf({\mu})={\omega}$. Then
$\Carr {\kappa}{\lambda}{\omega}\oot$ whenever $\oot\le
{\kappa}\le {\lambda}$.
\end{theorem}

\begin{proof}
By corollary \ref{cor:gen_step_up}, it clearly suffices to show
that if $\omega_2 < \lambda' \le \lambda$ and $\acal$ is any
${\omega}$-almost disjoint set-system of cardinality ${\lambda}'$,
then $\acal$ is $\omega_2$-cut by a ${\lambda}'$-chain of
elementary submodels.

To see this, let $\<M_{\alpha}:{\alpha}<{\lambda}'\>$ be any
${\lambda}'$-chain of elementary submodels satisfying $\oot\cup
\{\acal\}\subs M_1$ and for every ${\alpha}<{\lambda}'$ write
$N_{\alpha}=M_{{\omega}{\alpha}}$. We claim that
$\<N_{\alpha}:{\alpha}<{\lambda}'\>$, also a ${\lambda}'$-chain of
elementary submodels, $\oot$-cuts $\acal$.

Indeed, assume that $\alpha < \lambda'$ and $A\in \acal$ with
$|A\cap N_{\alpha}| \ge \oot$. Since ${\omega}{\alpha}$ is a limit
ordinal, then there is ${\beta}<{\omega}{\alpha}$ such that
$|A\cap M_{\beta}|\ge\oo$. But then $\acal' = \{A'\in
\acal:|A'\cap M_{\beta}|\ge \oo\}\in M_{{\beta}+1}$ and
$|\acal'|\le |M_{\beta}|$ by lemma \ref{lm:gen_small_close}, hence
we have $A \in \acal' \subs M_{{\beta}+1}\subs
M_{{\omega}{\alpha}}=N_{\alpha}$.
\end{proof}

A very short alternative proof of theorem \ref{tm:above_oot} may
be obtained as follows. In \cite[Theorem 6]{EH3} Erd\H os and
Hajnal proved that if ${\mu}^{\omega}={\mu}^+$ holds for each
${\mu}<{\lambda}$ with $\cf({\mu})={\omega}$ then
\begin{equation}\label{eh2x}
\text{$\MM {\lambda}{\kappa}{\omega}\to \BB {{\omega}_2}$ whenever
${\omega}_1\le {\kappa}\le {\lambda}$}.
\end{equation}
Moreover, under the same assumption,  Komj\'ath proved in
\cite[Theorem 5]{KO2} that
\begin{equation}\label{ko2}
\text{$\MM {\lambda}{{\omega}_2}{\omega}\to {\bf ED}\,$ for all
$\lambda \ge \omega_2 $.}
\end{equation}
Applying proposition \ref{lm:reduction} with $\mu = \omega$ and
$\tau = \omega_2$, we may conclude that (\ref{eh2x}) and
(\ref{ko2}) together imply $[\lambda,\kappa,\omega] \Rightarrow
\omega_2$ whenever $\lambda \ge \kappa \ge \omega_2$.

Actually, the above proof yields the stronger conclusion
$$[\lambda,\kappa,\omega] \Rightarrow \omega_1 \mbox{ whenever }\,\lambda
\ge \kappa \ge \omega_1\,,$$ provided that in (\ref{ko2}) we may
replace $\omega_2$ by $\omega_1$. But by \cite[Theorem 5(c)]{KO2},
this can be done if, in addition to ${\mu}^{\omega}={\mu}^+$ for
all ${\mu}<{\lambda}$ with $\cf({\mu})={\omega}$, we also assume
$\Box_\mu$ for each ${\mu}<{\lambda}$ with $\cf({\mu})={\omega} <
\mu$. (In fact, as it is shown in \cite{HJS1}, the assumption of a
very weak version of $\Box_\mu$ suffices for this.) Thus we get
the following result.

\begin{theorem}\label{tm:above_oo}
Let ${\lambda}$ be an uncountable cardinal and assume that
\begin{enumerate}[(i)]
\item  ${\mu}^{\omega}={\mu}^+$ for each cardinal ${\mu}<{\lambda}$ with
  $\cf({\mu})={\omega}$,
\item
$\Box_\mu$ holds for each singular cardinal ${\mu}<{\lambda}$ with
$\cf({\mu})={\omega}$.
\end{enumerate}
Then $\Carr {\kappa}{\lambda}{\omega}\oo$ holds whenever $\oo\le
{\kappa}\le {\lambda}$.
\end{theorem}

As condition (ii) of theorem \ref{tm:above_oo} is only relevant
for $\lambda > \aleph_\omega\,$, we immediately obtain the
following result.

\begin{corollary}\label{cor:box}
CH and $\,\omega_1 \le \kappa \le \lambda \le \aleph_\omega\,$
imply $\,\Carr {\kappa}{\lambda}{\omega}\oo$.
\end{corollary}

\section{
Consistent lower bounds for $\ccf(\lambda,\kappa,\omega)$ }

In the previous section we gave (consistent) universal upper
bounds for $\ccf(\lambda,\kappa,\omega)$ when $\kappa \ge
\omega_2$ and $\kappa \ge \omega_1$, respectively. That no such
universal upper bound can be given for
$\ccf(\lambda,\omega,\omega)$ follows from the fact that if
$\,\clubsuit(\lambda)$ holds, that is for each $\alpha \in
E_\omega^\lambda$ there is an $\omega$-type subset $A_\alpha$
cofinal in $\alpha$ such that for every $X \in [\lambda]^\lambda$
we have $A_\alpha \subs X$ for some $\alpha \in E_\omega^\lambda$,
then clearly $$ \ccf(\lambda,\omega,\omega) \ge
\chi(\lambda,\omega,\omega) \ge \cf(\lambda)\,.$$ In particular,
if $\lambda$ is also regular then we have $$
\ccf(\lambda,\omega,\omega) = \chi(\lambda,\omega,\omega) =
\lambda\,.$$

In order to get some lower bounds for
$\ccf(\lambda,\kappa,\omega)$ with $\kappa > \omega$, and thus to
show that the results of the previous section are sharp, we shall
make use of a result in \cite{HJS1}. First we give some notation.

If $\lambda > \omega_1$ is a regular cardinal and $S \subs
E_{\omega_1}^{\lambda}$ is stationary then we denote by $\prin
{S}$ the following statement:
\begin{itemize}
\item[{$\prin {S}$}:] {\em there is  an
${\omega}$-almost disjoint family $\{A_{\alpha}:{\alpha}\in S\}$
such that  $\,A_{\alpha}\,$ is a cofinal subset of $\,\alpha$ of
order type $\omega_1$ for each ${\alpha}\in S$.}
\end{itemize}

It is an immediate consequence of Fodor's pressing down theorem
that such an $\{A_{\alpha}:{\alpha}\in S\}$ is not essentially
disjoint, hence if we assume condition (i) of theorem
\ref{tm:above_oo} then (very weak) $\Box_\mu$ must fail at some
singular $\mu < \lambda$ with $\cf(\mu) = \omega$, in particular
$\lambda > \aleph_\omega$. This implies that if $\prin {S}$ holds
then we must have some large cardinals, and in fact it was shown
in \cite{HJS1} that the existence of a supercompact cardinal
implies the consistency of GCH with $\bigstar(S)$ for some $S
\subs \aleph_{\omega+1}$.

For any set $S \subs \lambda$ we shall denote by $\clubsuit(S)$
the statement that there is a sequence $\{B_{\alpha}:{\alpha}\in
S\}$ with $\cup B_{\alpha}={\alpha}$ for each ${\alpha}\in S$ such
that for every $X \in [\lambda]^\lambda$ we have $B_\alpha \subs
X$ for some $\alpha \in S$. Then $\{B_{\alpha}:{\alpha}\in S\}$ is
called a $\clubsuit(S)$-sequence. Clearly, every
$\diamondsuit(S)$-sequence is a $\clubsuit(S)$-sequence.

\begin{theorem}\label{tm:club}
Assume that ${\lambda} > 2^\omega$ is a regular cardinal and we
have both $\prin {S}$ and $\diamondsuit(S)$  for a stationary set
$S \subs E_{\omega_1}^{\lambda}$. Then

\noindent (1) there is an $\omega$-{\em almost disjoint}
$\clubsuit(S^*)$-sequence for some $S^*\subs S$, hence
$$\ccf(\lambda,\omega_1,\omega) = \chi(\lambda,\omega_1,\omega) =
\lambda\,;$$

\noindent (2) for every cardinal $\kappa$ with $\omega_2 \le
\kappa < \lambda$ we have $\omega_2 \le
\ccf(\lambda,\kappa,\omega)$.
\end{theorem}

\begin{proof}
(1) Let us fix an ${\omega}$-almost disjoint family
$\{A_{\alpha}:{\alpha}\in S\}$ witnessing $\prin {S}$ and a
$\diamondsuit(S)$-sequence $\{B_{\alpha}:{\alpha}\in S\}$. Let
$$B_\alpha = \{b(\alpha,\gamma) : \gamma < \tip(B_\alpha)\}$$ be the
increasing enumeration of $B_\alpha$.

Next, by transfinite recursion we define  sets
$\{E_{\alpha}:{\alpha}\in S\}$ as follows. Assume that
$\{E_{\beta}:{\beta}\in {\alpha}\cap S\}$ has been constructed. If
$\,\tip(B_{\alpha}) < {\alpha}$ then let $E_\alpha = \emptyset$.
Otherwise, if $\tip(B_{\alpha})={\alpha}$, set
\begin{displaymath}
 E_{\alpha}=\{b({\alpha},\gamma): \gamma \in A_\alpha \}\,,
\end{displaymath}
clearly then $E_\alpha \in [B_\alpha]^{\omega_1}$ is cofinal in
$\alpha$.

Let us next define
\begin{displaymath}
S^*=\{{\alpha}\in S : |E_{\alpha}|=\oo\land  \,\forall {\beta}\in
S\cap {\alpha}\,\,(\,|E_{\beta}\cap E_{\alpha}|<{\omega})\,\},
\end{displaymath}
and
\begin{displaymath}
  \ecal=\{E_{\alpha}:{\alpha}\in S^*\}.
\end{displaymath}
Then $\ecal\subs \br {\lambda};\oo;$ is ${\omega}$-almost disjoint
by definition and we claim that $\ecal$ is a
$\clubsuit(S^*)$-sequence.

Indeed, let $B\in \br \lambda;\lambda;$ and consider the club set
\begin{displaymath}
C=\{{\xi}<{\lambda}:\tip(B\cap {\xi})={\xi}\}
\end{displaymath}
and the stationary set
\begin{displaymath}
\hat S=\{{\alpha}\in S\cap C: B\cap {\alpha} =B_{\alpha}\}.
\end{displaymath}
Now, if $\alpha \in \hat S \cap S^*$ then $E_\alpha \subset
B_\alpha = B \cap \alpha \subs B$, hence it suffices to show that
$\hat S \cap S^* \ne \emptyset$.

Assume, on the contrary, that $\hat S \cap S^* = \emptyset$. Then
for each ${\alpha}\in \hat S$, as $\tip(B_{\alpha})={\alpha}$,
there is a $ {\beta}<{\alpha}$ such that $E_{\alpha}\cap
E_{\beta}$ is infinite. By Fodor's theorem and $2^\omega <
\lambda\,$, there are ${\beta}<{\alpha}<{\alpha'}$ and $X\in \br
E_{\beta};{\omega};$ such that $\alpha,\alpha' \in \hat S$ and
$X\subs E_{\alpha}\cap E_{\alpha'}$.  But $B_{\alpha}= \alpha \cap
B_{\alpha'}\,$, hence $b\big({\alpha},\gamma\big) =
b\big({\alpha'},\gamma\big)$ for all $\gamma < \alpha$ and
$b\big({\alpha'},\gamma\big) \notin B_\alpha$ for $\gamma \ge
\alpha$, consequently $x \in E_\alpha \cap E_{\alpha'}$ implies
that $x = b(\alpha,\,\gamma)$ for some $\gamma \in A_{\alpha}\cap
A_{\alpha'}$.  This, however contradicts $|A_{\alpha}\cap
A_{\alpha'}| < \omega$,  proving that $\hat S \cap S^* \ne
\emptyset$ and so $\ecal$ is a $\clubsuit(S^*)$-sequence.

But then $\ecal$ is a $(\lambda,\omega_1,\omega)$-system for which
$\,\ccf(\ecal) = \chi(\ecal) = \lambda\,$ holds trivially,
completing the proof of part (1).

\medskip

(2) Having fixed $\kappa$ with $\omega_1 < \kappa < \lambda$, we
shall construct a $(\lambda,\kappa,{\omega})$-system $\fcal\subs
\br {\lambda};\kappa;$ such that for every function
$h:{\lambda}\to \oo$ there is $F\in \fcal$ for which
\begin{equation}\label{eq:h}
\text{$\nu \in h[F]\,$ implies  $\,|F \cap h^{-1}\{{\nu}\}| \ge
\oo$}.
\end{equation}

Consider the club set $K = \{\kappa \cdot \xi : \xi < \lambda \}$
and, for every $\xi < \lambda$, let $K_\xi$ denote the
(half-closed) interval $\big[\kappa \cdot \xi\,, \kappa \cdot
(\xi+1) \big)$. We can assume, without any loss of generality,
that $S \subs K$.

For every $\alpha \in S$ we also fix a partition of $A_\alpha$
into $\oo$-many disjoint uncountable pieces: $A_\alpha =
\cup\{A_{\alpha,\nu}:{\nu}<\oo\}$. Finally, this time, we use
$\diamondsuit(S)$ by choosing a $\diamondsuit(S)$-sequence
$\{h_{\alpha}:{\alpha}\in S\}$ for the functions $h : {\lambda}
\to \oo$.

Next, by transfinite recursion define the sets
$\{E_{\alpha}:{\alpha}\in S\}$ as follows. Assume that $\alpha \in
S$, moreover $\{E_{\beta}:{\beta}\in {\alpha}\cap S\}$ has been
constructed. Let
\begin{displaymath}
D_{\alpha}=\{{\nu}<\oo: \tip(h_{\alpha}^{-1}\{{\nu}\})={\alpha}\},
\end{displaymath}
for every $\nu \in D_\alpha$ let $\,\{b({\alpha},{\nu},{\eta}):
{\eta}<{\alpha}\}$ be the increasing enumeration of
$h_{\alpha}^{-1}\{{\nu}\}\,$, and put
\begin{displaymath}
 E'_{\alpha}=\{b({\alpha},{\nu},\gamma):
{\nu}\in D_{\alpha}, \gamma \in A_{\alpha,\nu}\}.
\end{displaymath}
Of course, if $D_\alpha = \emptyset$ then we have $E'_\alpha =
\emptyset$ as well, and in this case we put $E_\alpha =
\emptyset$. If, however, $D_\alpha \ne \emptyset$ then for every
$\nu \in D_\alpha$ the set $B_{\alpha,\nu} =
\{b({\alpha},{\nu},\gamma): \gamma \in A_{\alpha,\nu}\}$ is
cofinal in $\alpha$. Thus, using that $\alpha = \kappa \cdot \xi$
for some $\xi$ with $\cf(\xi) = \omega_1 < \kappa$, we can find
$E_\alpha \subs E'_\alpha$ such that (i) $|E_\alpha \cap
B_{\alpha,\nu}| = \omega_1$ for each $\nu \in D_\alpha$, and (ii)
$|E_\alpha \cap K_\zeta| \le 1$ for every $\zeta < \lambda$.

Next, similarly as in the proof of (1), we let
\begin{displaymath}
S^*=\{{\alpha}\in S : |E_{\alpha}|=\oo\land  \,\forall {\beta}\in
S\cap {\alpha}\,\,(\,|E_{\beta}\cap E_{\alpha}|<{\omega})\,\},
\end{displaymath}
and then for any $\alpha = \kappa \cdot \xi  \in S^*$ we define
\begin{displaymath}
F_{\alpha}=E_{\alpha}\cup K_\xi = E_{\alpha}\cup
[\alpha,\,\alpha+\kappa)\,, \mbox{ and
}\fcal=\{F_{\alpha}:{\alpha}\in S^*\}.
\end{displaymath}
Clearly, $\fcal\subs \br {\lambda};\kappa;$ and $\fcal$ is
${\omega}$-almost disjoint because, by (ii), we have
$|F_{\alpha}\cap F_{\beta}|\le |E_{\alpha}\cap E_{\beta}|+1$ for
any $\{\alpha,\beta \} \in [S^*]^2$.

Now, consider any map $h:{\lambda}\to \oo$ and let
\begin{displaymath}
D=\{{\nu}<\oo:|h^{-1}\{{\nu}\}|={\lambda}\}\,;
\end{displaymath}
then $D\ne \empt$. For every ${\nu}\in D$ put
\begin{displaymath}
C_{\nu}=\{{\xi}<{\lambda}:\tip({\xi}\cap(h^{-1}\{{\nu}\}))={\xi}\}
\end{displaymath}
and
\begin{displaymath}
C=\cap\{(C_{\nu}:{\nu}\in D\},
\end{displaymath}
then $C$ is a club set.

We have ${\eta}=\sup (h^{-1}[\oo\setm D]) < \lambda\,$ because
$\lambda > \omega_1$ is regular. Let $T=S \cap C \setm {\eta}$,
then $h[T]\subs D$,
$$\hat S=\{{\alpha}\in T: h\restriction {\alpha} =h_{\alpha}\}$$
is stationary, and  if ${\alpha}\in \hat S$ then $ D_\alpha = D$.

Note that if $\alpha \in \hat S \cap S^*$ then $h[F_{\alpha}] =
h[E_{\alpha}]=D_{\alpha}=D$ and, by our construction,
$$|h^{-1}\{{\nu}\}\cap E_{\alpha}| = \omega_1$$ for each
${\nu}\in D$, hence $F_{\alpha} \in \fcal$ witnesses (\ref{eq:h}).
Thus, to prove part (2), it again suffices to show that $\hat S
\cap S^* \ne \emptyset$.

Assume, on the contrary, that $\hat S \cap S^* = \emptyset$. Since
$D_{\alpha}=D\ne \empt$ for every $\alpha \in \hat S \subs C$ this
would imply that for every ${\alpha}\in \hat S$ there exists
${\beta}<{\alpha}$ for which $E_{\alpha}\cap E_{\beta}$ is
infinite. But then, in the same way as in the proof of (1), we
could conclude that there is a pair $\{{\alpha},{\alpha'}\} \in
[\hat S]^2$ with $\alpha < \alpha'$ such that $E_{\alpha}\cap
E_{\alpha'}$ is infinite. Using that $h_\alpha = h_{\alpha'}
\upharpoonright \alpha$ and hence $h_{\alpha}^{-1}\{{\nu}\}$ is an
initial segment of $h_{\alpha'}^{-1}\{{\nu}\}$, this would imply
that $A_{\alpha}\cap A_{\alpha'}$ is also infinite, a
contradiction.
\end{proof}

As we noted above, it was shown in \cite{HJS1} that the existence
of a supercompact cardinal implies the consistency of GCH with
$\bigstar(S)$ for some $S \subs E^{\aleph_{\omega+1}}_{\omega_1}$.
This, together with theorem \ref{tm:club}, immediately yields the
following result which shows that the results of the previous
section are sharp, even under GCH.

\begin{corollary}\label{cor:hsj}
If it is consistent that there is a supercompact cardinal then it
is also consistent that GCH holds and
\begin{enumerate}[(1)]
\item    $\chi(\aleph_{\omega+1},\omega_1,\omega) = \ccf(\aleph_{\omega+1},\omega_1,\omega ) = \aleph_{\omega+1}$,
\item $\ccf(\aleph_{\omega+1},\omega_n,\omega) =\oot$ for  $2\le n\le{\omega}$.
\end{enumerate}
\end{corollary}

We conclude this section with a (somewhat surprising) result
showing that consistently, e.g. under GCH, the relation
$\chi(\lambda,\omega_1,\omega) \le \omega_1$, hence
$\ccf(\lambda,\omega_1,\omega) \le \omega_1$ as well, implies $\MM
{\lambda}{{\omega}_1}{\omega}\to {\bf ED}\,$.

\begin{theorem}\label{tm:Rr}
Let ${\lambda}$ be an uncountable cardinal and assume that
\begin{enumerate}[(i)]
\item  ${\mu}^{\omega}={\mu}^+$ for any ${\mu}<{\lambda}$ with
$\cf({\mu})={\omega}$,

\smallskip

\item
if $\omega < {\mu}<{\lambda}$ with $\cf({\mu})={\omega}$ then
$\diamondsuit(S)$ holds for every stationary set $\,S \subs
E_{\omega_1}^{\,\mu^+}\,$.
\end{enumerate}
Then $\,\chi(\lambda,\omega_1,\omega) \le \omega_1$ implies $\MM
{\lambda}{{\omega}_1}{\omega}\to {\bf ED}\,$.
\end{theorem}

\begin{proof}
We shall prove this by induction on ${\lambda}$. It is trivially
true for $\lambda = \omega_1$, hence we can assume ${\lambda}>\oo$
and that it holds for all ${\lambda}'<{\lambda}$.

We shall make use of the following obvious corollary of our
assumption (i): If $X$ is any set with $|X| \le \lambda$ and
$\fcal \subs [X]^{\omega_1}$ is $\omega$-almost disjoint then
$|\fcal| \le |X|$. In fact, this follows from the following
consequence of (i): $\,\mu^\omega = \mu$ if $\mu \le \lambda\,$
with $\,\cf(\mu)
> \omega$.

Now, let  $\acal\subs \br {\lambda};\oo;$  be  an
${\omega}$-almost disjoint set-system, we have to show that
$\acal$ is essentially disjoint.

\medskip

\noindent{\bf Case 1:} {\em $\lambda$ is a limit cardinal or
${\lambda} = {\mu}^+$ for some ${\mu}$ with $\cf({\mu})
> {\omega}$.}

\smallskip

Condition (i) implies $\nu^\omega < \lambda$ for any $\nu <
\lambda$, hence we can find a ${\lambda}$-chain
$\<M_{\alpha}:{\alpha}<{\lambda}\>$ of elementary submodels with
$\acal \in M_1$ and $\omega_1 \subs M_1$ such that $\br
M_{{\alpha}};{\omega};\subs M_{{\alpha}+1}$ for each
${\alpha}<{\lambda}$. Let us put
$N_{\alpha}=M_{{\omega}\cdot{\alpha}}$ for ${\alpha}<{\lambda}$ ,
then $\acal$ is $\omega_1$-cut by the $\lambda$-chain
$\<N_{\alpha}:{\alpha}<{\lambda}\>$.

Indeed, if $|A\cap N_{\alpha}|=|A\cap M_{{\omega}\cdot{\alpha}}| =
\oo\,$ then there is a ${\beta}<{\omega}\cdot{\alpha}$ such that
$|A\cap M_{\beta}|\ge{\omega}$. Since $\acal$ is ${\omega}$-almost
disjoint and $\br M_{{\beta}};{\omega};\subs M_{{\beta}+1}$ then
we have $A\in M_{{\beta}+1} \subs
M_{{\omega}\cdot{\alpha}}=N_{\alpha}\,$.

For ${\alpha}<{\lambda}$ let
\begin{displaymath}
  \acal_{\alpha}=\acal\cap (N_{{\alpha}+1}\setm N_{\alpha})\,,
\end{displaymath}
then $|\acal_{\alpha}|\le |N_{{\alpha}+1}|<{\lambda}$. By this and
the inductive hypothesis there is a function
$F_{\alpha}:\acal_{\alpha}\to \br {\lambda};{\omega};$ such that
$A \cap N_\alpha \subs F_{\alpha}(A)$ for all $A \in \acal_\alpha$
and the family
\begin{displaymath}
  \{A\setm F_{\alpha}(A):A\in \mc A_{\alpha}\}
\end{displaymath}
is disjoint. Now, it is easy to check that the function $$F =
\cup_{\alpha < \lambda}F_\alpha :\mc A\to \br
{\lambda};{\omega};$$ witnesses the essential disjointness of
$\acal$.

\medskip

\noindent{\bf Case 2:} {\em ${\lambda}= {\tau}^+$  for some
singular cardinal ${\tau}$ with $\cf({\tau})={\omega}$.}

\smallskip

For any $A\in \acal$ let
\begin{equation}
L(A)=\{{\alpha}<{\lambda}:\cf({\alpha})=\oo\, \mbox{ and
}\,{\alpha}= \sup A\cap {\alpha}\}.
\end{equation}
Clearly, then $1\le |L(A)|\le \oo$. We claim that the set
  \begin{displaymath}
   S=\cup\{L(A):A\in \acal\}
  \end{displaymath}
is non-stationary in ${\lambda}$.

Indeed, by definition, for each $A\in\acal$ we may find a family
of pairwise disjoint sets
$$\{B(A,{\alpha}):{\alpha}\in L(A)\}\subs \br A;\oo;$$  such that
$\sup (B(A,{\alpha}))={\alpha}$ and $\tip(B(A,{\alpha})=\oo$. So,
if $S$ were stationary then the $\omega$-almost disjoint family
$$\bcal=\{B(A,{\alpha}):A\in \acal, {\alpha}\in L(A)\}$$ would
witness $\prin {S}\,$. But then, by condition (ii) and part (1) of
theorem \ref{tm:club}, we would have
$\,\chi(\lambda,\omega_1,\omega) = \lambda > \omega_1$, a
contradiction. So there is a club $E\subs {\lambda}$ such that
\begin{equation}
E\cap \cup\{L(A):A\in \acal\}=\empt .
\end{equation}

It follows from our introductory remark that if
${\delta}<{\lambda}$ then
\begin{equation}\label{eq:claim}
\big|\{A\in \acal:|A\cap {\delta}| = \oo\}\big|\le
{\delta}<{\lambda}\,,
\end{equation}
hence the following set $D$ is club in ${\lambda}$:
\begin{equation}
D=\{{\zeta}<{\lambda}: \forall {\delta}<{\zeta}\ \forall A\in
\acal \ (\text{ $|A\cap {\delta}| = \oo$ implies $A\subs
{\zeta}$})\}.
\end{equation}
Let $C=E\cap D$ and  $C=\{{\gamma}_{\nu}:{\nu}<{\lambda}\}$ be the
increasing enumeration of $C$.

For any $A\in \acal$ let $${\nu}_A=\min \{{\nu}<{\lambda}: |A\cap
{\gamma}_{\nu}| = \oo\}\,.$$ Then $C \subs E$ implies that
${\nu}_A$ can not be a limit ordinal, hence  ${\nu}_A={\eta}_A+1$.
This and the definition of $D$ imply
\begin{equation}
|A\cap {\gamma}_{{\eta}_A}|\le {\omega}\, \text{  and  }  \,A\subs
{\gamma}_{{\eta}_A+1}.
\end{equation}

Let us put $\acal_{\eta}=\{A\in \acal:{\eta}_A={\eta}\}\,$ for any
$\eta < \lambda\,$, then $|\acal_\eta| \le \gamma_{\eta+1} <
\lambda$. By the inductive hypothesis, for each ${\eta}<{\lambda}$
there is a function $F_{\eta}:\acal_{\eta}\to \br
{\lambda};{\omega};$ such that $A \cap \gamma_\eta \subs
F_\eta(A)$ for any $A \in \acal_\eta$ and the family
\begin{displaymath}
  \{A\setm F_{\eta}(A):A\in \mc A_{\eta}\}
\end{displaymath}
is disjoint. Now, it is again easy to check that the function $$F
= \cup_{\eta < \lambda}F_\eta :\mc A\to \br {\lambda};{\omega};$$
witnesses the essential disjointness of $\acal$.
\end{proof}

\medskip

Let us remark that, by a recent result of Shelah from \cite{Sh},
if $\omega = \cf(\mu) < {\mu}$ and $2^\mu = \mu^+$ then
$\diamondsuit(S)$ holds for every stationary set $\,S \subs
E_{\omega_1}^{\,\mu^+}\,$. Consequently, conditions (i) and (ii)
of theorem \ref{tm:Rr} together are equivalent with the following
single statement: For all $\mu < \lambda$ with $\cf(\mu) = \omega$
we have $2^\mu = \mu^+$.

If $\acal$ is an essentially disjoint
$(\lambda,\omega_1,\omega)$-system then, trivially, we have
$\chi(\acal) = 2$, moreover there is a coloring $f : \cup \acal
\to \omega_1$ that satisfies $|\omega_1 \setm I_f(A)| < \omega_1$
for all $A \in \acal$. Consequently, from theorem \ref{tm:Rr} we
immediately obtain the following result.

\begin{corollary}\label{cor:el}
Under the assumptions of theorem \ref{tm:Rr}, in particular under
GCH, the following five statements are equivalent for an
uncountable cardinal $\lambda$:

\begin{enumerate}[1)]
\item $\,[\lambda,\omega_1,\omega] \Rightarrow \omega_1\,,$

\item$\,\ccf(\lambda,\omega_1,\omega) \le \omega_1\,,$

\item$\,\chi(\lambda,\omega_1,\omega) \le \omega_1\,,$

\item$\,\chi(\lambda,\omega_1,\omega) = 2\,,$

\item$\,\MM
{\lambda}{{\omega}_1}{\omega}\to {\bf ED}\,$.

\end{enumerate}
\end{corollary}

\bigskip

\section{On $\ccf({\omega_1},\omega_1,\omega)$ and  $\ccf({\omega_1},\omega,\omega)$}

\medskip

Our previous results give no help in deciding the exact values of
$\ccf({\omega_1},\omega_1,\omega)$ and
$\ccf({\omega_1},\omega,\omega)$, except proposition \ref{pr:kko}
which implies that both are equal to either $\omega$ or
$\omega_1$. We shall show below that actually both equal
$\omega_1$ under CH and both equal $\omega$ under $MA_{\aleph_1}$.
We also remark that, as any $(\omega_1,\omega_1,\omega)$-system
clearly has an $\omega$-witness, we have
\begin{equation}\label{eq:le}
\omega \le \ccf({\omega_1},\omega_1,\omega) \le
\ccf({\omega_1},\omega,\omega) \le \omega_1
\end{equation}
in ZFC. However, we do not know if their equality is provable in
ZFC.

That CH implies $\ccf({\omega_1},\omega,\omega) = \omega_1$ is an
immediate consequence of the following ZFC result of Komj\'ath
\cite{KO}.

\begin{theorem}
  $$\chi({\cont},\omega,\omega) =
  \cont\,.$$
\end{theorem}

Before giving our proof that CH also implies
$\ccf({\omega_1},\omega_1,\omega) = \omega_1$, we need a
preparatory lemma.

\begin{lemma}\label{lm:eqx}
Let $\acal \subs \br \oo;\oo;$ be $\omega$-almost disjoint and
$\ical(\acal)$ be the ideal generated by $\acal$, that is, $X \in
\ical(\acal)$ iff there is $\bcal \in [\acal]^{<\omega}$ with $X
\subs \cup \bcal$. Then, for any $X \subs \omega_1\,$, $\,X \cap
\alpha \in \ical(\acal)$ for all $\alpha < \omega_1$ implies $X
\in \ical(\acal)$.
\end{lemma}
\begin{proof}
For each ${\alpha}<\oo$ we may pick a $\subs$-minimal
$\bcal_\alpha \in [\acal]^{<\omega}$ such that $X\cap {\alpha}
\subs^* \cup \bcal_\alpha$, i.e. $|X \cap {\alpha} \setm \cup
\bcal_\alpha| < \omega$. There is $I\in \br \oo;\oo;$ for which
$\{\bcal_{\alpha}:{\alpha}\in I\}$ forms a $\Delta$-system with
kernel $\bcal$. We claim that $\bcal_{\alpha}=\bcal$ for all
${\alpha}\in I$. Then we are done because this implies $X \subs^*
\cup \bcal$ and hence $X \in \ical(\acal)$ by $X \subs \cup
\acal$.

So assume, on the contrary, that $\alpha \in I$ and $A \in
\bcal_\alpha \setm \bcal$. By the $\subs$-minimality of
$\bcal_\alpha$ then $$Y = A \cap (X \cap \alpha \setm \cup
\bcal)$$ must be infinite. But, for any $\beta \in I$ with $\beta
> \alpha$, if $B \in \bcal_\beta \setm \bcal$ then $|B \cap Y| \le |B \cap A| <
\omega$, contradicting $$Y \subs X \cap \beta \setm \cup \bcal
\subs^* \cup (\bcal_\beta \setm \bcal).$$
\end{proof}

\begin{theorem}\label{tm:ch_w1w1}
CH implies $$\ccf({\omega_1},\omega_1,\omega) =
\ccf({\omega_1},\omega,\omega) = \omega_1\,.$$
\end{theorem}

\begin{proof}
By induction on ${\alpha}$, we shall construct an
${\omega}$--almost disjoint family
$\acal=\{A_{\alpha}:{\alpha}<\oo\}\subs \br \oo;\oo;$ such that
for any coloring $h:\oo\to {\omega}$ there is $A_\alpha \in \acal$
satisfying
\begin{equation}\label{eq:oooo}
 \forall \, n \in h[A_\alpha]\,\,\big(\,|h^{-1}\{n\}\cap A_\alpha|={\omega}
\big).
\end{equation}
To start with, using CH, let
\begin{itemize}
\item $\{T_{\alpha}:{\alpha}<\oo\}$ be a partition of $\oo$ into
uncountable sets such that $T_{\alpha}\subs \oo\setm {\alpha}\,$
for every $\alpha < \omega_1$;
\item $\{S_{\alpha}:{\alpha}<\oo\}$ be an enumeration of
$[\omega_1]^\omega$.
\end{itemize}

Assume that $\{A_{\beta}:{\beta}<{\alpha}\}$ has been constructed
and we have $\alpha \in T_\gamma$. For any subset $a\subs
{\alpha}$ we write $\as a =\cup\{A_{\beta}:{\beta}\in a\}$, in
particular, $\as \xi = \cup_{\eta <\xi}A_\eta$. Consider the set
\begin{displaymath}
H_{\alpha}=\{{\beta}<{\alpha}: S_\beta \subs \alpha \setm \as
\gamma\, \mbox{ and } \,\forall a\in \br {\alpha};<{\omega};\
\big|S_{\beta}\setm \as a \big|={\omega}\}.
\end{displaymath}

We can choose $B_{\alpha} \subs {\alpha} \setm \as \gamma$ such
that
\begin{enumerate}
\item
$|B_{\alpha}\cap A_\beta|<{\omega}$ for each $\beta < \alpha$,

\medskip

\item
$|B_{\alpha} \cap S_\beta|={\omega}\,$ whenever $\,{\beta}\in
H_{\alpha}\,$.
\end{enumerate}
Indeed, if $H_\alpha = \empt$ then $B_\alpha = \empt$ works, and
otherwise $B_\alpha$ can be obtained by a simple recursive
construction. Finally, let us put $A_{\alpha}=B_{\alpha}\cup
T_{\alpha}$. Note that, by definition, $A_\beta \cap A_\alpha =
A_\beta \cap B_\alpha$ is finite for every $\beta < \alpha$.

Let $A = \as {\,\omega_1} = \cup \acal$ and consider any coloring
$h:A \to {\omega}$. We set
\begin{displaymath}
I =\{n \in {\omega}:\,\exists\,\delta < \omega_1\, (\,h^{-1}\{n\}
\subs \as \delta\,) \}
\end{displaymath}
and  $K = \omega \setm I$. We may then find $\gamma < \omega_1$
such that $h^{-1}(I) \subs \as \gamma$.

For any $n\in K$ consider the set
\begin{displaymath}
X_n=h^{-1}\{n\}\setm \as {\gamma}\,,
\end{displaymath}
then obviously $X_n \notin \ical(\acal)\,$. Thus, by  lemma
\ref{lm:eqx}, there is ${\alpha}_n<\oo$ such that $X_n \cap
\alpha_n \notin \ical(\acal)\,$ as well.  For each $n \in K$ pick
$\beta_n<\oo$ with $S_{\beta_n}=X_n\cap {\alpha}_n$  and choose
${\alpha}\in T_{\gamma}$ such that ${\alpha}>\sup\{\beta_n:n\in
K\}$.

Clearly, then $\{\beta_n:n\in K\}\subs H_{\alpha}$, hence
$$B_{\alpha}\cap h^{-1}\{n\}\supset B_{\alpha}\cap (X_n\cap
{\alpha}_{n}) = B_{\alpha}\cap S_{\beta_n}$$ is infinite for every
$n \in K$. If, however, $n\in I$ then $h^{-1}\{n\}\subs \as
{\gamma}$, and so $B_\alpha \cap \as \gamma = A_\alpha \cap \as
\gamma = \empt$ implies $h^{-1}\{n\}\cap A_{\alpha}=\empt$. Thus
$A_{\alpha}$ witnesses (\ref{eq:oooo}).
\end{proof}

Now we turn to our other promised result, namely that
$MA_{\omega_1}$ implies $$\ccf({\omega_1},\omega_1,\omega) =
\ccf({\omega_1},\omega,\omega) = \omega\,.$$ In fact, we prove the
following stronger theorem.

\begin{theorem}\label{tm:ma1}
If $MA_{\omega_1}$ holds then

\noindent (1) $\,[\,{\omega_1},\omega,\omega] \Rightarrow \omega$,

\smallskip

\noindent (2) $\,[\,{\omega_1},\omega_1,\omega] \Rightarrow
\omega$.
\end{theorem}

\begin{proof}
Let us start by noting that (2) follows from (1) because every
$({\omega_1},\omega_1,\omega)$-system admits an $\omega$-witness.

Now, to prove (1), let us consider any
$({\omega_1},\omega,\omega)$-system $\acal = \{A_\alpha : \alpha <
\omega_1\} \subs [\omega_1]^\omega$. We then define a poset
$\pcal = \<P,\preceq\>$ as follows. Let $P = Fn(\omega_1,\omega)
\times [\omega_1]^{<\omega}$ and for $\<f,I\>, \<g,J\>\in P$ put
$\<g,J\>\preceq \<f,I\>$ iff $g \supset f\,,J \supset I\,$, and
for all ${\alpha}\in I$ we have

\begin{enumerate}[(i)]
\item $(g\setm f)\restriction A_{\alpha}$ is 1--1, and

\smallskip

\item $(g\setm f)[A_{\alpha}] \cap f[A_\alpha] = \empt$.
\end{enumerate}
It is easy to check that $\preceq$ is indeed a partial order on
$P$.

We next show that $\pcal$ is CCC. To see this, consider first two
members of $P$, say $p = \<f,I\>$ and $q = \<g,J\>$, such that the
following conditions hold with $D = \dom f$ and $E = \dom g\,$:

\begin{enumerate}[(a)]
\item $f \upharpoonright D \cap E = g \upharpoonright D \cap E\,$,
i.e. $f$ and $g$ are compatible functions;

\smallskip

\item $A[I] \cap (E \setm D) = \empt = A[J] \cap (D \setm E)\,$.
\end{enumerate}
(Here, as in the proof of theorem \ref{tm:ch_w1w1}, $\as x = \cup
\{A_\alpha : \alpha \in x \}$.) Then, trivially, $r = \< f \cup g,
I \cup J \> \in P$ and $r \preceq p,q$. Indeed, for instance, $r
\preceq p$ because $(g \setm f) \upharpoonright A_\alpha = \empt$
for each $\alpha \in I$. Thus, to show that $\pcal$ is CCC, it
will suffice to prove that among any $\omega_1$ members of $P$
there are two that satisfy (a) and (b).

\newcommand{\fel}[1]{\<f_{#1}, I_{#1}\>}

So let $\{p_{\nu}:{\nu}<\oo\}\subs P$ with $p_{\nu}=\fel {\nu}$.
Using standard $\Delta$-system and counting arguments  we can
assume the following:
\begin{enumerate}[1)]
\item $\{\dom(f_{\nu}):{\nu}<\oo\}$ forms a $\Delta$-system
with kernel $D$ and we have $D<D_{\nu}<D_{\mu}$ for
${\nu}<{\mu}<\oo$, where $D_{\nu}=\dom(f_{\nu})\setm D$.

\smallskip

\item\label{iii} $f_{\nu}\restriction D=f$ and $|D_{\nu}|=n$ for all $\nu <
\omega_1$.

\smallskip

\item $\{I_{\nu}:{\nu}<\oo\}$ forms a $\Delta$-system
with kernel $I$ and $I<J_{\nu}<J_{\mu}$ for ${\nu}<{\mu}<\oo$,
where $J_{\nu}=I_{\nu}\setm I$. Moreover, $|J_{\nu}|=m$ for all
$\nu < \omega_1$.

\smallskip

\item $\as I < D_0$ and $\as {I_\nu} < D_\mu$ whenever $\nu < \mu <
\omega_1$.
\end{enumerate}

\begin{claim}\label{cl:NM}
If $N\in \br \oo;{\omega};$ and $M\in \br \oo;n \cdot m+1;$
satisfy $\, N < M\,$ then there are ${\nu}\in N$ and ${\mu}\in M$
such that $D_{\nu}\cap \as {J_{\mu}} =\empt$.
\end{claim}

\begin{proof}[Proof of the claim]
Let $\ucal$ be a non-principal ultrafilter on $N$. Write
$D_{\nu}=\{{\delta}_{{\nu},i}:i<n\}$ and
$J_{\mu}=\{{\alpha}_{{\mu},j}:j<m\}$.

Assume, on the contrary, that for any ${\nu}\in N$ and ${\mu}\in
M$ there are $i<n$ and $j<m$ such that ${\delta}_{{\nu}, i}\in
A_{{\alpha}_{{\mu},j}}$. This implies that, for any fixed $\mu \in
M$, there is a pair $\<i,j\> \in n \times m$ for which
$$V^{i,j}_\mu = \{\nu \in N : {\delta}_{{\nu}, i}\in
A_{{\alpha}_{{\mu},j}}\} \in \ucal\,.$$ Then, as $|M| > n \cdot
m$, there are two distinct $\mu,\,\mu' \in M$ and a pair $\<i,j\>
\in n \times m$ such that both $V^{i,j}_\mu \in \ucal$ and
$V^{i,j}_{\mu'} \in \ucal$ and hence  $V^{i,j}_{\mu} \cap
V^{i,j}_{\mu'} \in \ucal$ is infinite. This, however, would imply
that
$$A_{\alpha_\mu,j} \cap A_{\alpha_{\mu',j}} \supset\, \{\delta_{\nu,i} : \nu \in   V^{i,j}_{\mu} \cap V^{i,j}_{\mu'}\}$$
is also infinite, a contradiction.
\end{proof}

But if $\nu,\,\mu$ are as in claim \ref{cl:NM}, then clearly (a)
and (b) are satisfied for $p_\nu$ and $p_\mu$, and hence they are
compatible. This completes the proof that $\pcal$ is CCC.

Let us now consider, for every $\alpha < \omega_1$ and $n <
\omega$,  the sets
\begin{displaymath}
D_{\alpha}=\{\<f,I\>\in P:{\alpha}\in \dom (f)\,\},
\end{displaymath}
and
\begin{displaymath}
E^n_{\alpha}=\{\<f,I\>\in P:{\alpha}\in I \mbox{ and } n \in
f[A_\alpha]\}.
\end{displaymath}
It is easy to check that all these sets are dense in $\pcal$, let
us only do it for the $E^n_{\alpha}$. Indeed,  any $\<f,I\>\in P$
is extended by $\<f,I \cup \{\alpha\}\>\,$, so we may assume that
$\alpha \in I$, to begin with. Now, if $n \notin \ran(f)$ then
pick first $\gamma \in A_\alpha \setm \as {I \setm \{\alpha\}}\,$.
Obviously, we have then $\<f \cup \{\<\gamma,n \>\},I\> \preceq
\<f,I\>$ and $\<f \cup \{\<\gamma,n \>\},I\> \in E^n_{\alpha}$.

By $MA_{\omega_1}$ there is a filter $\gcal$ in $\pcal $ that
meets all the dense sets $D_\alpha$ and $E^n_{\alpha}$. Let us put
\begin{displaymath}
F=\cup\{f:\<f,I\>\in \gcal\}.
\end{displaymath}
Then $F : \omega_1 \to \omega$ because $\gcal$ meets every
$D_\alpha$ and we claim that $I_F(A_\alpha) =^* \omega$ for each
$\alpha < \omega_1$. Indeed, $\gcal \cap E^n_{\alpha} \ne \empt$
for all $n < \omega$ implies $F[A_\alpha] = \omega$. Moreover,
there is some $\<f,I\>\in \gcal$ with $\alpha \in I$,
consequently, by the definition of $\preceq$, we clearly have
$I_F(A_\alpha) \supset \omega \setm f[A_\alpha]$.
\end{proof}

\smallskip

\begin{problem}
Is $\,\ccf(\omega_1,\omega_1,\omega) =
\ccf(\omega_1,\omega,\omega)\,$ provable in ZFC?
\end{problem}

Recall that ``stick" is the following combinatorial statement, a
common weakening of CH and $\clubsuit = \clubsuit(\omega_1)$:
There is a family $\acal \subs [\omega_1]^\omega$ such that
$|\acal| = \omega_1$ and for every uncountable set $S \subs
\omega_1$ we have an $A \in \acal$ with $A \subs S$. We know that
stick implies $\,\ccf(\omega_1,\omega,\omega) = \omega_1$.

\begin{problem}
Does stick imply $\,\ccf(\omega_1,\omega_1,\omega) = \omega_1$?
\end{problem}

\begin{problem}
Is $\,\ccf(2^\omega,2^\omega,\omega) = 2^\omega\,$ provable in
ZFC?
\end{problem}

\end{document}